\documentclass[11pt, a4paper]{article}
\usepackage{amsmath, amssymb, amsthm,mathtools}
\usepackage{empheq}
\usepackage{mathrsfs}
\usepackage{geometry}
\usepackage{graphicx}
\usepackage{hyperref}
\usepackage{enumitem}
\usepackage[numbers]{natbib} 
\usepackage{parskip}
\setcitestyle{square}

\usepackage{mathrsfs}
\usepackage{bm}
\usepackage{xcolor}

\newtheorem{theorem}{Theorem}[section]
\newtheorem{lemma}[theorem]{Lemma}
\newtheorem{proposition}[theorem]{Proposition}
\newtheorem{corollary}[theorem]{Corollary}
\newtheorem{definition}[theorem]{Definition}

\numberwithin{equation}{section}
\usepackage[capitalize]{cleveref}

\newcommand{\cL}{\mathcal{L}}
\newcommand{\cM}{\mathcal{M}}
\newcommand{\cT}{\mathcal{T}}
\newcommand{\cH}{\mathcal{H}}
\newcommand{\R}{\mathbb{R}}
\newcommand{\C}{\mathbb{C}}
\newcommand{\Z}{\mathbb{Z}}
\newcommand{\N}{\mathbb{N}}
\newcommand{\ka}{\kappa}
\newcommand{\eps}{\varepsilon}
\newcommand{\cP}{\mathcal{P}}
\newcommand{\cA}{\mathcal{A}}

\newcommand{\cC}{\mathcal{C}}
\newcommand{\cG}{\mathcal{G}}
\newcommand{\cI}{\mathcal{I}}
\newcommand{\cB}{\mathcal{B}}
\newcommand{\noopsort}[1]{}

\title{Instability of periodic waves in generalized Korteweg--de Vries--Burgers equation with monostable source}
\author{Aarushi Kansal \& Ashish Kumar Pandey \\ \textit{IIIT Delhi} \\ \texttt{aarushik@iiitd.ac.in, ashish.pandey@iiitd.ac.in}}
\date{}

\begin{document}

\maketitle

\begin{abstract}
We study periodic traveling waves in a general class of nonlinear dispersive equations with Burgers dissipation and a monostable Fisher--KPP reaction source. Under natural assumptions on the dispersion symbol, we prove the existence of small-amplitude periodic traveling waves bifurcating from a constant equilibrium by means of a Lyapunov--Schmidt reduction. We then investigate the spectral stability of these waves by combining Floquet--Bloch decomposition with perturbation theory for linear operators. It is shown that the Floquet spectrum of the linearized operator intersects the unstable complex half-plane, implying spectral instability with respect to localized perturbations. The results recover the recently established instability theory for the
Korteweg--de Vries--Burgers--Fisher equation as a special case and
provide a unified framework for a broad class of local and nonlocal
dispersive systems. These results reveal a common instability mechanism generated by the interplay of dispersion, Burgers dissipation, and monostable reaction dynamics.
\end{abstract}

\section{Introduction}\label{S1:Intro}
We investigate the existence and spectral stability properties of spatially periodic traveling wave solutions to the generalized  Korteweg--de Vries--Burgers (KdVB) equation augmented with a monostable reaction source term, governed by the following Partial Differential Equation (PDE)
\begin{equation}\label{S1:eq1}
u_{t} + \cM u_{x} + \alpha u u_{x} = u_{xx} + r u(1-u),
\end{equation}
where $u = u(x,t) \in \R$ is the state variable representing the wave profile at position $x \in \R$ and time $t > 0$. The parameter $\alpha \in \R \setminus \{0\}$ scales the strength of the convective nonlinearity, while $r > 0$ denotes the intrinsic growth rate of the monostable reaction source. The term $u_{xx}$ represents a standard parabolic, Burgers-type kinematic dissipation \cite{Burgers1948}. The operator $\cM$ is a self-adjoint, nonlocal linear pseudodifferential operator acting on the spatial domain, formally defined via its Fourier transform by
\begin{equation}\label{S1:eq2}
\widehat{\cM f}(\xi) = m(\xi) \widehat{f}(\xi),
\end{equation}
where $m(\xi)$ is referred to as the symbol or the Fourier multiplier of the dispersion operator. 

To ensure the existence of periodic traveling wave solutions to \eqref{S1:eq1}, we impose several assumptions on the Fourier multiplier $m(\xi)$: 
\begin{enumerate}[label={(A\arabic*)}]
    \item The function $m: \R \to \R$ is continuous, even ($m(-\xi) = m(\xi)$ for all $\xi \in \R$), and analytic in $\R\backslash \{0\}$.\label{assumption1} 
    \item The symbol $m(\xi)$ possesses a sublinear or algebraic growth/decay at infinity; specifically, there exist constants $M>0$, $C_1 > 0$, $C_2>0$ and $\beta \ge -1$ such that $C_1|\xi|^\beta \leq |m(\xi)| \le C_2 |\xi|^\beta$ for all $|\xi|> M$.\label{assumption2}
    \item The symbol $m(\xi)$  satisfies strict inflection properties on the positive real axis $\R^+$ ensuring that resonance conditions are nondegenerate. Specifically, we exclude  wavenumbers $\kappa$ for which $m(\kappa)= m(\kappa n)$ for some $n\in \N$, $n> 1$.\label{assumption3}
\end{enumerate}

\subsection{Related Dispersive Model Equations}
Different choices of the nonlocal operator $\cM$ or its associated Fourier multiplier $m(\xi)$ in \eqref{S1:eq1} provide an extensive array of classical and contemporary models in fluid mechanics and mathematical physics:
\begin{enumerate}
   \item \textbf{The Korteweg--de Vries (KdV) Equation:}
Choosing $$m(\xi)=-\xi^2,$$ yields the classical local dispersion operator
$\cM=\partial_x^2$. Originally formulated by Korteweg and de Vries
\cite{KdV1895} to describe the propagation of long, weakly nonlinear
surface gravity waves in shallow water, the KdV equation has become one
of the fundamental models in nonlinear wave theory. It is completely
integrable, possesses infinitely many conservation laws, and admits a rich
family of coherent structures including solitary waves and periodic
traveling waves (cnoidal waves) \cite{Whitham1974, Ablowitz_Clarkson_1991}. The existence,
spectral properties, and nonlinear stability of periodic KdV waves have
been extensively investigated and serve as a benchmark for the study of
periodic waves in dispersive systems
\cite{Gardner1993,BronskiJohnsonKapitula2010}.
When augmented with Burgers-type dissipation, the equation reduces to
the classical KdVB equation, originally
derived by Su and Gardner \cite{SuGardner1969} under weak nonlinearity
and long-wave scaling assumptions. The KdVB equation represents one of
the simplest physically relevant models exhibiting the simultaneous
interaction of nonlinear convection, dispersion, and dissipation, and
has been applied to the study of plasma waves, undular bores, viscous
shallow-water flows, and turbulence phenomena
\cite{Johnson1970,JeffreyKakutani1972}. 
\item \textbf{The Whitham Equation:}
Choosing
$$m(\xi)=\sqrt{\frac{\tanh(\xi)}{\xi}},$$
recovers the Whitham equation introduced by Whitham
\cite{Whitham1967,Whitham1974} as a nonlinear model retaining the
exact linear dispersion relation for surface water waves while
preserving the quadratic nonlinearity of the KdV equation. In contrast
to KdV, the Whitham equation combines full linear dispersion with
nonlinearity and is therefore capable of capturing phenomena beyond
the reach of long-wave asymptotic models. The equation has attracted
considerable attention in recent years due to its ability to describe wave breaking, the existence of highest and cusped traveling waves, and modulational instability mechanisms analogous to the Benjamin--Feir instability observed in water waves
\cite{Hur2015_breaking,JohnsonHur2015,EhrnstromWahlen2019}.
Periodic traveling waves of the Whitham equation have been studied
extensively, and their spectral and modulational stability properties
are now relatively well understood
\cite{EhrnstromWahlen2019,Hur2015_stab,BronskiHurJohnson2016}.
\item \textbf{The Benjamin--Ono (BO) Equation:}
Choosing $$m(\xi)=-|\xi|,$$
yields the nonlocal dispersion operator
$\cM=-\cH\partial_x$, where $\cH$ denotes the Hilbert transform.
The resulting Benjamin--Ono equation was independently derived by
Benjamin \cite{Benjamin1967_BO} and Ono \cite{Ono1975} as a model
for the propagation of long internal waves along a sharp density
interface in a deep stratified fluid. In contrast to the KdV equation, the BO equation possesses a genuinely nonlocal dispersive mechanism,
yet remains completely integrable and admits both solitary and periodic
traveling wave solutions. Owing to the presence of the Hilbert transform,
the spectral and dynamical properties of BO waves differ substantially
from those of local dispersive equations and have been the subject of
extensive mathematical investigation
\cite{Angulo2009,Matsuno1979}. Dissipative perturbations of BO lead to
the Benjamin--Ono--Burgers equation, an
important prototype for the interaction between nonlocal dispersion,
viscosity, and dispersive shock formation \cite{Matsuno2007BOBurgers}. 
\item \textbf{The Fractional Korteweg--de Vries (fKdV) Equation:}
Choosing
$$ m(\xi)=-|\xi|^\beta,
\qquad \beta\in(0,2],$$
yields the family of fractional Korteweg--de Vries equations with
nonlocal dispersive operator $\cM=-\Lambda^\beta$, where
$\Lambda=(-\partial_x^2)^{1/2}$. These models interpolate between the
Benjamin--Ono equation $(\beta=1)$ and the classical KdV equation
$(\beta=2)$, providing a natural framework for studying the influence
of dispersion strength and nonlocality on nonlinear wave propagation.
The fKdV equation has attracted considerable attention in connection
with wave breaking, global regularity, and the existence and stability
of traveling waves. In particular, Johnson \cite{Johnson2013}
established the existence of small-amplitude periodic traveling waves
and derived modulational instability criteria depending explicitly on the dispersion exponent $\beta$, revealing qualitative transitions in the stability behavior as the degree of nonlocality varies. Dissipative perturbations of fractional KdV equations, obtained by
combining fractional dispersion with Burgers-type diffusion, have also
been investigated as models for the interaction between nonlocal
dispersion and viscous effects \cite{Biler2001}.
\item \textbf{The Intermediate Long Wave (ILW) Equation:}
Choosing
$$m(\xi)=\xi\coth(\delta\xi)-\frac{1}{\delta},
\qquad \delta>0,$$
yields the Intermediate Long Wave (ILW) equation, a nonlocal dispersive
model describing the propagation of long internal waves in a stratified
fluid of finite depth \cite{Joseph1977,KubotaKoDobbs1978}.
The parameter $\delta$ represents the relative depth of the fluid layer,
and the ILW equation provides a continuous transition between the
Korteweg--de Vries and Benjamin--Ono equations through the asymptotic
limits $\delta\to0$ and $\delta\to\infty$, respectively. Like the BO
equation, the ILW equation is completely integrable and possesses both
solitary and periodic traveling wave solutions. The existence,
stability, and qualitative properties of periodic ILW waves have been
studied extensively, with explicit formulations in terms of elliptic
functions playing a central role in their analysis
\cite{Pava2017,Angulo2009}. The dissipative ILW equation obtained by incorporating the Burgers-type diffusion provides an intermediate dispersive-dissipative model
connecting the KdV--Burgers and BO--Burgers regimes.
\item \textbf{The Kawahara Equation:}
Choosing $$m(\xi)=-\xi^2+\gamma\xi^4,$$
with $\gamma>0$, yields the Kawahara equation, a higher-order dispersive model introduced by Kawahara \cite{Kawahara1972} to describe the propagation of weakly nonlinear long waves when the leading-order third-order dispersion becomes small, and fifth-order dispersive effects
must be retained. The competition between third and fifth-order
dispersion leads to richer wave dynamics than those observed in the KdV
equation, including oscillatory solitary waves and intricate periodic
traveling wave patterns. The existence, spectral stability, and
modulational stability of periodic traveling waves for the Kawahara equation and related fifth-order dispersive models have been studied extensively in recent years \cite{Haragus2011,TrichtchenkoDeconinckKollar2018}.
\end{enumerate}

\subsection{Physical Relevance of the Monostable Source Term}
The inclusion of the monostable reaction term $ru(1-u)$ in
\eqref{S1:eq1} fundamentally alters the nature of the underlying
dynamics. In contrast to purely dispersive or dispersive-dissipative
models, which primarily describe the transport, redistribution, or dissipation of energy, the reaction term introduces an internal
mechanism for growth and energy production. Consequently, the resulting equation belongs to the broader class of reaction-convection-dispersion
systems and may be viewed as describing an active medium whose dynamics
are driven not only by transport and wave propagation but also by local
self-amplification processes.

The nonlinear source term $u(1-u)$ was introduced independently by
Fisher \cite{Fisher1937} and by Kolmogorov, Petrovsky, and Piskunov
\cite{KPP1937} in their pioneering studies of the spatial spread of
advantageous genetic traits. The resulting Fisher--KPP equation became
one of the foundational models in the theory of reaction-diffusion
systems and has since found applications in population dynamics,
epidemiology, combustion theory, chemical kinetics, ecological invasion,
tumor growth, and pattern formation. More generally, monostable reaction
terms describe systems possessing two equilibrium states, an unstable
state $u=0$ and a stable state $u=1$, with the nonlinear growth
mechanism driving solutions away from the former and toward the latter.

A hallmark of Fisher-KPP dynamics is the emergence of traveling fronts connecting the unstable and stable equilibria. The existence,
propagation, and stability of such fronts have been studied extensively for reaction-diffusion equations and their extensions; see, for
instance, \cite{Fisher1937,KPP1937,Murray2002,Volpert1994}. Reaction-dispersion systems often exhibit behaviors that are absent in their conservative counterparts, including oscillatory invasion fronts, wave generation behind propagating interfaces, and destabilization of otherwise stable wave patterns \cite{Maini2006}. 

From a mathematical perspective, the monostable source term destroys many of the conservation laws and continuous symmetries present in classical dispersive equations. Consequently, the continuous families of traveling and periodic waves characteristic of integrable models are often replaced by isolated branches selected through nonlinear balance
between dispersion, dissipation, and reaction. This selection mechanism
has significant implications for spectral stability and bifurcation structure \cite{vanSaarloos2003}. 

The present work is motivated by the observation that many physically relevant dispersive models, including the KdV, Whitham, Benjamin--Ono,
fractional KdV, Intermediate Long Wave, and Kawahara equations, admit natural dissipative extensions through Burgers diffusion. It is therefore natural to investigate how the additional presence of a monostable reaction term influences the existence and stability of periodic traveling waves across this broader class of dispersive
systems.

\subsection{Periodic Traveling Wave Setting and Main Theorems}
Periodic traveling waves play a fundamental role in nonlinear wave theory, serving as building blocks for more complicated wave patterns
and providing insight into the long-time dynamics of dispersive media. In the present setting, the simultaneous presence of nonlocal dispersion, Burgers dissipation, and a monostable reaction term creates a delicate balance that may generate nontrivial periodic wave trains through bifurcation from a uniform equilibrium state.

To seek such solutions, we introduce the traveling-wave ansatz
\begin{equation}\label{S1:eq3}
u(x,t)=w(z),\qquad z=\kappa(x-ct),
\end{equation}
where $c\in\mathbb R$ denotes the wave speed and $\kappa>0$ is the fundamental wavenumber. The profile function $w$ is assumed to be
$2\pi$-periodic in $z$, so that the corresponding solution is $\frac{2\pi}{\kappa}$-periodic in the spatial variable $x$. Substituting \eqref{S1:eq3} into \eqref{S1:eq1} transforms the problem into a
nonlinear nonlocal ordinary differential equation for the periodic profile $w$.

Our first result establishes the existence of a branch of small-amplitude periodic traveling waves bifurcating from the trivial equilibrium
solution. The bifurcation occurs near the distinguished parameter values $$
(\kappa_0,c_0)=\bigl(\sqrt r,m(\sqrt r)\bigr),$$ which arise from the linearization of the profile equation and identify the onset of periodic behavior.
\vspace{10pt}
\begin{theorem}[Existence and Asymptotics of Periodic Traveling Waves]\label[theorem]{S1:thm1}
Suppose assumptions $\mathrm{(A1)-(A3)}$ hold. Then, there exists an open neighborhood of the critical parameters $(\kappa_0, c_0) = (\sqrt{r}, m(\sqrt{r}))$ and a small parameter $\eps_0 > 0$ such that for each $\eps \in (0, \eps_0)$, there exists a unique (up to translation) nontrivial solution $u(x,t)=w(\kappa (\eps)(x-c(\eps)t))=:w^\eps(z)$ to \eqref{S1:eq1} with wave speed $c(\eps)$ and wavenumber $\kappa(\eps)$, where $w$ is $2\pi$-periodic and smooth in its argument. Moreover, $w$, $\ka$, and $c$ are analytic at $\eps=0$ with their leading order expansions:
\begin{align}
w^\eps(z) &= \eps \cos(z) + \eps^2\left(\frac12+A_2 \cos(2z)+B_2\sin(2z)\right) + \mathcal{O}(\eps^3), \label{S1:eq4}\\
\kappa(\eps) &= \ka_0 + \eps^2 \kappa_2 + \mathcal{O}(\eps^3), \label{S1:eq6}\\
c(\eps) &= c_0 + \eps^2 c_2 + \mathcal{O}(\eps^3), \label{S1:eq5}
\end{align}
where
\begin{align*}
A_2 & = -\frac{3r^2 + 2 \alpha \ka_0^2 (m(2\ka_0)-c_0)}{2(9r^2 + 4\ka_0^2 (m(2\ka_0)-c_0)^2)}, \quad
B_2  = \frac{3 \alpha r \ka_0 - 2 r \ka_0 (m(2\ka_0)-c_0)}{2(9r^2 + 4\ka_0^2 (m(2\ka_0)-c_0)^2)},
\end{align*}
\begin{align*}
\ka_2 &= -\frac{r(1+A_2) + \frac{1}{2}\alpha \ka_0 B_2}{2\ka_0},  \quad
c_2 = \frac{(\alpha\ka_0-rm'(\ka_0) )(1+A_2) - (\frac{1}{2}\alpha \ka_0m'(\ka_0) +2r) B_2}{2\ka_0}.
\end{align*}
\end{theorem}
\cref{S1:thm1} provides a complete local description of the bifurcating periodic waves, including explicit asymptotic expansions for the profile, wave speed, and wavenumber. A natural question is whether these waves persist under small perturbations of the governing
partial differential equation. Our second main result shows that this is not the case: the periodic waves constructed above are spectrally unstable with respect to localized perturbations.
\vspace{10pt}
\begin{theorem}[Spectral Instability]\label{S1:thm2}
For sufficiently small $\eps$, the periodic traveling wave $w^\eps$ of \eqref{S1:eq1} obtained in \cref{S1:thm1} is spectrally unstable with respect to perturbations in $L^2(\R)$. 
\end{theorem}
Since the hypotheses of Theorems \ref{S1:thm1} and \ref{S1:thm2} are satisfied by a broad class of dispersive symbols, the existence and instability theory developed here applies directly to dissipative-reaction counterparts of several classical nonlinear wave equations, including the KdV, Whitham, Benjamin--Ono, fractional KdV, Intermediate Long Wave, and Kawahara equations. Specific
corollaries for these models are presented in \S\ref{section6}.

\subsection{Relevance and Comparison with Existing Literature}
The first rigorous analysis of periodic traveling waves in a dispersive-dissipative equation with a Fisher--KPP source appears to be the recent work of Folino, Naumkina, and Plaza \cite{Folino2024}, who considered the Korteweg--de Vries--Burgers--Fisher equation. They proved the existence of small-amplitude periodic traveling waves bifurcating from a constant equilibrium through a local Hopf mechanism and established their spectral instability using Floquet theory and perturbation arguments. Owing to the local nature of the KdV dispersion operator, the associated traveling wave equation can be
reduced to a finite-dimensional ordinary differential equation, making available the tools of dynamical systems and local bifurcation theory.

The present work extends this theory far beyond the classical KdV setting. Rather than focusing on a single dispersive equation, we consider a broad class of dispersive symbols $m(\xi)$ encompassing local, nonlocal, fractional, and higher-order dispersive models. This includes dissipative-reaction counterparts of the KdV, Whitham, Benjamin--Ono, fractional KdV, Intermediate Long Wave, and Kawahara equations. To our knowledge, there are currently no general existence and instability results for periodic traveling waves in dispersive-dissipative equations with monostable reaction terms at this level of generality. This transition from a local to a nonlocal framework renders the standard dynamical systems techniques entirely inapplicable and creates severe technical challenges:
\begin{itemize}
\item The pseudodifferential operator $\cM$ prevents the reduction of the traveling-wave equation to a finite-dimensional dynamical system. Consequently, classical phase-plane methods, center-manifold theory, and standard Hopf bifurcation arguments are no longer available. 
\item The linearized profile operator is genuinely nonlocal and acts on periodic Sobolev spaces. Establishing its Fredholm properties, kernel structure, and bifurcation geometry requires a careful Fourier multiplier analysis that depends explicitly on the properties of the symbol $m(\xi)$. 
\item The existence theory must therefore be developed entirely within an infinite-dimensional functional-analytic framework. This necessitates a constructive Lyapunov--Schmidt reduction adapted to nonlocal operators together with explicit asymptotic expansions of the bifurcating wave family. 
\item The spectral problem is likewise nonlocal and must be analyzed through Floquet--Bloch decomposition and perturbation theory for families of pseudodifferential operators. 
\end{itemize}
By establishing the existence and spectral instability of small-amplitude periodic traveling waves under general assumptions on the dispersion symbol, we show that the destabilizing influence of a monostable Fisher--KPP source persists across a remarkably broad class of dispersive systems. Our results recover the Korteweg--de Vries--Burgers--Fisher theory of Folino et al.~\cite{Folino2024} as a special case while simultaneously providing a unified framework for nonlocal, fractional, and higher-order dispersive equations. In this sense, the present work identifies a common instability mechanism generated by the interplay of dispersion, Burgers dissipation, and monostable reaction dynamics.

The remainder of the paper is organized as follows.
\S\ref{Section2} introduces notation and preliminary results. \S\ref{section3} develops the periodic traveling wave formulation and proves \cref{S1:thm1} via a Lyapunov--Schmidt reduction.
\S\ref{section4} formulates the spectral problem and derives the associated
Floquet--Bloch decomposition.
\S\ref{section5} establishes \cref{S1:thm2} using perturbation theory for the Bloch operators. Finally, \S\ref{section6} discusses applications of the theory to several classical dispersive models. Auxiliary results are collected in an appendix.

\section{Preliminaries}\label{Section2}
Let $X$ and $Y$ be Banach spaces, then $B(X,Y)$ is the space of all bounded linear operators from $X$ to $Y$. In particular, $B(X)$ is the space of all bounded linear operators from $X$ to $X$.

Let $p\in[1,\infty]$, then $L^{p}(\R)$ denotes the space of all Lebesgue measurable, real or complex-valued functions over $\R$ such that 
\begin{equation*}
\|f\|_{L^{p}(\R)}\coloneqq\left(\int\limits_{-\infty}^{\infty}|f(x)|^{p}\, dx\right)^{\frac{1}{p}}<\infty \quad \text{if} \quad p<\infty,
\end{equation*}
and $\|f\|_{L^{\infty}(\R)}\coloneqq\text{ess sup}_{x\in\R}|f(x)|<\infty$ if $p=\infty$.

Let $p\in[1,\infty]$, then $L^{p}_{\text{per}}(\mathbb{T})$ denotes the space of all $2\pi$-periodic, Lebesgue measurable, real or complex-valued functions over the one-dimensional Torus $\mathbb{T} = \mathbb{R}/2\pi\mathbb{Z}$ such that 
\begin{equation*}
\|f\|_{L^{p}_{\text{per}}(\mathbb{T})}\coloneqq\left(\frac{1}{2\pi}\int\limits_{0}^{2\pi}|f(x)|^{p}\, dx\right)^{\frac{1}{p}}<\infty \quad \text{if} \quad p<\infty,
\end{equation*}
and $\|f\|_{L^{\infty}_{\text{per}}(\mathbb{T})}\coloneqq\text{ess sup}_{x\in[0,2\pi]}|f(x)|<\infty$ if $p=\infty$.

The space $C^{\infty}(\R)$ is the space of all smooth (infinitely-differentiable) functions defined on $\R$. Also, $C^{\infty}_{\text{per}}(\mathbb{T})$ is the space of smooth  and $2\pi$-periodic functions.

The Schwartz Space $\mathcal{S}(\R)$ is the space of rapidly decreasing functions defined by
\begin{equation*}
\left\{f\in C^{\infty}(\R)\Bigm|\lim\limits_{|x|\to\infty}|x^k  f^{(j)}(x)| =0 \text{ for all integers } k, j\right\}.
\end{equation*}

Let $m>0$ be an integer and let $1\leq p \leq \infty$ then we define
\begin{equation*}
W^{m,p}(\R)\coloneqq\{f\in L^{p}(\R)\,|f^{(j)}\in L^{p}(\R) \text{ for all integers } j,1\leq j\leq m\}.
\end{equation*}

For $s\in\R$, 
\begin{equation*}
H^{s}(\R)\coloneqq\left\{f\in \mathcal{S}'(\R)\,\Bigm|\left(1+|\xi|^{2}\right)^\frac{s}{2}\widehat{f}(\xi)\in L^{2}(\R)\right\},
\end{equation*}
such that,
\begin{equation*}
\|f\|_{H^{s}(\R)}\coloneqq\left(\int\limits_{-\infty}^{\infty}\left(1+|\xi|^{2}\right)^s|\widehat{f}(\xi)|^{2}\, d\xi\right)^{\frac{1}{2}}<\infty,
\end{equation*}
and $H^{\infty}(\R)\coloneqq\bigcap\limits_{s\in\mathbb{N}\cup\{0\}}H^{s}(\R)$ and where $\mathcal{S}'(\R)$ is the space of all tempered distributions on $\R$.

For $s\in\R$, \begin{equation*}
H^{s}_{\text{per}}(\mathbb{T})\coloneqq\left\{f\in \mathcal{D}'_{\text{per}}(\mathbb{T})\,\Big| \sum_{n\in \Z}\left(1+|n|^{2}\right)^\frac{s}{2}\widehat{f}(n)<\infty\right\},
\end{equation*}
such that,
\begin{equation*}
\|f\|^{2}_{H^{s}_{\text{per}}(\mathbb{T})}\coloneqq\sum\limits_{n\in\mathbb{Z}}(1+|n|^{2})^{s}|\widehat{f}(n)|^{2},
\end{equation*}
and $H^{\infty}_{\text{per}}(\mathbb{T})\coloneqq\bigcap\limits_{s\in\mathbb{N}\cup\{0\}}H^{s}_{\text{per}}(\mathbb{T})$ and where $\mathcal{D}'_{\text{per}}(\mathbb{T})$ is the space of periodic distributions.

Let H be a separable Hilbert space and $\langle X,\mu\rangle$ be a measure space. Then $L^2(X;H)$ is the
Hilbert space of square integrable $H$-valued functions with norm defined as 
\begin{equation*}
\|f\|^{2}_{L^2(X;H)}\coloneqq\int_{X}\|f(x)\|_{H}^{2}\,d\mu(x).
\end{equation*}

We consider the orthonormal basis of $L^{2}_{\text{per}}(\mathbb{T})$ as $\left\{e^{inz}:n\in\Z\right\}$. Any $f\in L^{1}_{\text{per}}(\mathbb{T})$ is expressed in the Fourier series as 
\begin{equation*}
f(z)\sim\sum\limits_{n\in\Z}\widehat{f}(n)e^{inz}, \quad \text{where } \widehat{f}(n)\coloneqq\frac{1}{2\pi}\int\limits_{0}^{2\pi}f(z)e^{-inz}\, dz.
\end{equation*}
If $f\in L^{p}_{\text{per}}(\mathbb{T})$, $p>1$, the Fourier series converges to $f$ pointwise almost everywhere. If $f\in$ $H^{s}_{\text{per}}(\mathbb{T})$, then the Fourier series is absolutely convergent to $f$, provided $s>\frac{1}{2}$.

We define the $L^{2}_{\text{per}}(\mathbb{T})$-inner product as 
\begin{equation*}
\langle f,g\rangle_{L^{2}_{\text{per}}(\mathbb{T})}\coloneqq\frac{1}{2\pi}\int\limits_{0}^{2\pi}f(z)\overline{g}(z)\, dz=\sum\limits_{n\in\Z}\widehat{f}(n)\overline{\widehat{g}(n)}.
\end{equation*}
For $f\in L^{1}(\R)$ we define its Fourier transform as 
\begin{equation*}
\widehat{f}(\xi)\coloneqq\frac{1}{2\pi}\int\limits_{-\infty}^{\infty}f(x)e^{-i\xi x}\,dx,
\end{equation*}
and we define the inverse Fourier transform (given $\widehat{f}\in L^{1}(\R)$) as 
\begin{equation*}
f(x)\coloneqq\int\limits_{-\infty}^{\infty}\widehat{f}(\xi)e^{i\xi x}\,d\xi.
\end{equation*}
Let $f(x)$ and $g(x)$ be two functions then $f(x)$ is said to be $\mathcal{O}(g(x))$ if there exist positive constants $C$ and $x_0$ such that
\begin{equation*}
    |f(x)|\leq C |g(x)| \quad \text{for all } x\geq x_0.
\end{equation*}
\begin{theorem}[Analytic Implicit Function Theorem]\label[theorem]{S2:thm1}
Let $X, Y$ and $Z$ be Banach spaces, and let $U \subset X \times Y$ be an open neighborhood of a point $(x_0, y_0)$. Suppose that a mapping $F: U \to Z$ is analytic with $F(x_0, y_0) = 0$ and the partial Fr\'echet derivative $\partial_x F(x_0, y_0) : X \to Z$ is a bounded linear isomorphism.
Then, there exist neighborhoods $V \subset X$ of $x_0$ and $W \subset Y$ of $y_0$, and a unique analytic mapping $x^*: W \to V$ such that $F(x^*(y), y) = 0$ for all $y \in W$. 
\end{theorem}
\begin{proof}
Refer \citep{Buffoni2003}.
\end{proof}
\begin{theorem}[Analytic Lyapunov--Schmidt Reduction]\label[theorem]{S2:thm2}
Let $X, Y$ and $\Lambda$ be Banach spaces (where $\Lambda$ represents parameter space). Let $F$ be an analytic mapping from $X \times \Lambda \to Y$. Let $U\subset X$ and $V\subset \Lambda$ be open where $F(x_{0}, \lambda_{0})=0$ for some $(x_{0}, \lambda_{0})\in U\times V$. Suppose the Fr\'echet derivative $L_0=\partial_{x}F(x_{0}, \lambda_{0})$ is a Fredholm operator of index zero and $\ker(L_0)\neq \left\{0\right\}$. Let $X = \ker(L_0) \oplus X_2$ and $Y=\operatorname{range}(L_0) \oplus Y_2$, where $Y_2 \cong \ker(L_0)$. Let $\Pi: Y \to Y_2$ be a continuous projection. Then there exists a neighbourhood $U_{2}\times V_{2}$ of $(x_{0}, \lambda_{0})$ in $U\times V \subset X\times \Lambda$ such that the problem $$F(x,\lambda)=0 \quad \text{for} \quad (x,\lambda)\in U_{2}\times V_{2},$$ is equivalent to the finite-dimensional problem 
\begin{equation*}
\Phi(u, \lambda) \coloneqq \Pi F(u + v^*(u, \lambda), \lambda) = 0,
\end{equation*}
where $u \in \ker(L_0)$, and $v^*(u, \lambda) \in X_2$ is a uniquely determined analytic function mapping a neighborhood of $(x_0, \lambda_0)$ into $X_2$.
\end{theorem}
\begin{proof}
Refer \citep{Kielhofer2004}.
\end{proof}

\section{Existence and asymptotics of Periodic Traveling Waves}\label{section3}
We aim to show the existence of small-amplitude periodic traveling wave solutions to \eqref{S1:eq1}. To achieve this we first substitute \eqref{S1:eq3} in \eqref{S1:eq1} which yields
\begin{equation}\label{S3:eq4}
-\ka c w'(z) + \ka \mathcal{M}_{\ka} w'(z) + \alpha \ka w(z) w'(z) = \ka^{2} w''(z) + r w(z)(1-w(z)),
\end{equation}
where $'=\frac{d}{dz}$. Here the bounded linear operator \begin{equation*}
\mathcal{M}: H^s(\mathbb{R}) \to H^{s-\beta}(\mathbb{R}) \quad \text{for all } s \geq 0,
\end{equation*}
defined via its Fourier transform on $\R$ as in \eqref{S1:eq2} is transformed to the following bounded linear operator on $\mathbb{T}$
\begin{equation*}
\cM_{\ka}:H^s_{\text{per}}(\mathbb{T})\to H^{s-\beta}_{\text{per}}(\mathbb{T}) \quad \text{for all }\ka>0 \text{ and for all } s\geq 0,
\end{equation*}
defined as
\begin{equation*}
\cM_{\ka}e^{inz}=m(\ka n)e^{inz} \quad \text{for } n\in\Z\setminus\left\{0\right\},
\end{equation*}
with $\cM_{\ka}(1)=m(0)$, see \cref{A} for a proof.

\subsection{Defining the Nonlinear Operator}
From \eqref{S3:eq4} we define the nonlinear operator $F$ as
\begin{equation}\label{S3:eq7}
F(w, \ka, c) \coloneqq \ka^{2} w'' - \ka(\cM_{\ka} - c)w' - \alpha \ka w w' + r w(1-w).
\end{equation}
The linear portion of $F$ involves the viscous term $\ka^{2} w''$ (order $2$) and the dispersive term $\cM_{\ka}w'$ (which, by the growth condition in Assumption \ref{assumption2}, is of pseudodifferential order $\beta + 1$). Thus, the maximal pseudodifferential order of the equation is given by $\mu := \max\{2, \beta+1\}$. Therefore, we define our domain space $X$ and co-domain space $Y$ as
\begin{equation}\label{S3:eq8}
X = H^{s+\mu}_{\text{per}}(\mathbb{T}) \quad \text{and} \quad Y = H^s_{\text{per}}(\mathbb{T}) \quad \text{for } s > 1/2.
\end{equation}
The restriction $s > 1/2$ ensures that $H^s_{\text{per}}(\mathbb{T})$ forms a Banach algebra. Hence, $$F:X\times \R^{+}\times \R\to Y$$ and we want to find $w$, $\ka$ and $c$ such that
\begin{equation}\label{S3:eq9}
F(w,\ka,c)=0.
\end{equation}
\begin{lemma}[The operator $F$ is analytic]\label[lemma]{S3:analytic}
Under Assumptions $\mathrm{\ref{assumption1}}$ and $\mathrm{\ref{assumption2}}$, the mapping $F: X \times \R^+ \times \R \to Y$ is well-defined and is analytic.
\end{lemma}
\begin{proof}We analyze $F$ term by term. The map $(w, \ka) \mapsto \ka^2 w''$ is a bounded linear operator in $w$ and a bounded bilinear operator in $\ka$ and therefore is an analytic map. The maps $(w, \ka, c) \mapsto \ka cw'$ and $w \mapsto rw$ are bounded, multilinear and linear operators, respectively and hence are analytic. The operator $w \mapsto \ka \cM_\ka w'$ is linear in $w$. Its dependence on $\ka$ is determined by the symbol $\ka \cdot m(\ka n) \cdot (in)$. Because $m$ is analytic in $\R^+$ by Assumption \ref{assumption1}, the mapping $\ka \mapsto \ka m(\ka n)$ is analytic in $\R^+$. The growth bound $|m(\ka n)| \le C_2|\ka n|^\beta$ by Assumption \ref{assumption2} guarantees that $\cM_\ka w'$ maps $H^{s+\mu}_{\text{per}}(\mathbb{T})$ to $H^{s+\mu-\beta-1}_{\text{per}}(\mathbb{T}) \subseteq H^s_{\text{per}}(\mathbb{T})$ boundedly. Because the Sobolev space $H^s_{\text{per}}(\mathbb{T})$ is a Banach algebra, the maps $(w,\ka) \mapsto -\alpha \ka w w'$ and $w \mapsto -rw^2$ are bounded bilinear maps in $w$ and hence are analytic. Since $F$ is a finite sum of analytic maps, $F$ is analytic on its domain.
\end{proof}
For any $\ka>0$ and $c\in\R$, $w\equiv 0$ and $w\equiv 1$ are the only constant solutions to \eqref{S3:eq9}. We shall seek nonconstant solutions to \eqref{S3:eq9} bifurcating from the constant solutions $w\equiv 0$ and $w\equiv 1$. The Fr\'echet derivative of $F$ in \eqref{S3:eq7}
with respect to $w$ evaluated at a point $(w,\ka,c)$ is given by
\begin{equation}\label{S3:eq10}
\partial_{w}F\left(w,\ka,c\right)\coloneqq\ka^{2}\frac{d^{2}}{dz^{2}}-\ka \left(\cM_{\ka}-c\right)\frac{d}{dz}-\alpha\ka w\frac{d}{dz}-\alpha\ka w'\cI+r(1-2w)\cI,
\end{equation}
where $\cI$ is the identity operator.

\subsection{Bifurcation for \texorpdfstring{$w\equiv 1$}{w equiv 1}}
At $w=1$, \eqref{S3:eq10} gives 
\begin{equation*}
\partial_{w}F\left(1,\ka,c\right)=\ka^{2}\frac{d^{2}}{dz^{2}}-\ka \left(\cM_{\ka}-c\right)\frac{d}{dz}-\alpha\ka\frac{d}{dz}-r\cI.
\end{equation*}
Let $g(z)\in X$ is in $\ker(\partial_{w}F(1,\ka,c))$, then $(\partial_{w}F(1,\ka,c))g(z)=0$. Using the Fourier series of $g$, $g(z)=\sum\limits_{n\in\Z}\widehat{g}(n)e^{inz}$, with $\widehat{g}(n)=\widehat{g_1}(n)+i\widehat{g_2}(n)$ we have 
\begin{equation}\label{S3:eq100}
\sum\limits_{n\in\Z}(\lambda_1(n)+i\lambda_2(n))(\widehat{g_1}(n)+i\widehat{g_2}(n))e^{inz}=0,
\end{equation}
where  
$$\lambda_1(n)=-\ka^{2}n^{2}-r \quad \text{and} \quad \lambda_2(n)=-\ka n\,m(\ka n)+c\ka n-\alpha\ka n.$$ 
Comparing the real and imaginary parts in \eqref{S3:eq100}, we get
$$\lambda_1(n)\widehat{g_1}(n)-\lambda_2(n)\widehat{g_2}(n)=0 \quad \text{and} \quad \lambda_2(n)\widehat{g_1}(n)+\lambda_1(n)\widehat{g_2}(n)=0,$$
which can be expressed as 
\begin{equation}\label{S3:eq101}
\begin{pmatrix}
\lambda_1(n) & -\lambda_2(n)\\ \lambda_2(n) & \lambda_1(n)
\end{pmatrix}
\begin{pmatrix}
\widehat{g_1}(n) \\ \widehat{g_2}(n)
\end{pmatrix}=
\begin{pmatrix}
    0\\0
\end{pmatrix}.
\end{equation}
For the above system \eqref{S3:eq101} to have a nontrivial solution, we must have $\lambda_1^{2}(n)+\lambda_2^{2}(n) = 0$, which means that $\lambda_1(n)=0$ and $\lambda_2(n)=0$. Therefore, $$-\ka^{2}n^{2}-r=0 \quad \text{and } -\ka n\, m(\ka n)+c\ka n-\alpha\ka n=0,$$  
which is not possible as $r > 0$. So, $\widehat{g}(n)=0$, for all $n\in\Z$ implying $g(z)=0$ and $\ker(\partial_{w}F(1,\ka,c))$ is trivial. Hence, by the Implicit Function Theorem (\cref{S2:thm1}), there is no bifurcation for $w\equiv 1$.  

\subsection{Bifuraction for \texorpdfstring{$w\equiv 0$}{w equiv 0}}
At $w=0$, \eqref{S3:eq10} gives 
\begin{equation}\label{S3:eq12}
\partial_w F(0,\ka,c) = \ka^{2}\frac{d^2}{dz^2} - \ka(\mathcal{M}_{\ka} - c)\frac{d}{dz} + r\cI.
\end{equation}
Let us define the linearization of $F$ around $w=0$ as a bounded linear operator $\cL(\ka,c) : X \to Y$ defined by
\begin{equation}\label{S3:eq13}
\cL(\ka,c)w \coloneqq \ka^{2} w'' - \ka(\mathcal{M}_{\ka} - c)w' + rw.
\end{equation}
Applying the operator in \eqref{S3:eq13} to the Fourier basis functions, we get for all $n\in\Z$,
\begin{align*}
\cL(\ka,c)e^{inz} &= \left( -\ka^2 n^2 - \ka(m(\ka n) - c)in + r \right)e^{inz}.
\end{align*}
The function $e^{inz}$ will be in $\ker(\cL(\ka,c))$ if
\begin{align*}
r - \ka^2 n^2 &= 0 \implies \ka = \frac{\sqrt{r}}{|n|}, \quad (\text{since } \ka > 0) \\
\text{and } - \ka(m(\ka n) - c)n &= 0 \implies c = m(\ka n).
\end{align*}
For any $n\in\N$, taking a general $g(z)\in X$ and using its Fourier series, it is straightforward to see that 
\begin{equation}\label{S3:eq14}
\ker\left(\cL\left(\frac{\sqrt{r}}{n},m(\ka n)\right)\right) = \operatorname{span}\{e^{-inz},e^{inz}\}.
\end{equation}
We choose $n=1$ to demonstrate the existence of the fundamental harmonic. The analysis remains the same for any fixed $n > 1$, $n\in\N$. For $n=1$, we set $\ka_0 = \sqrt{r}$ and $c_0 = m(\ka_0)=m(\sqrt{r})$ and the linearized operator at $w_0 = 0$ as 
\begin{equation}\label{S3:eq15}
\cL_{0}\coloneqq \cL(\ka_0, c_0)=\ka_{0}^{2}\frac{d^2}{dz^2} - \ka_0(\mathcal{M}_{\ka_0} - c_0)\frac{d}{dz} + r\cI.
\end{equation}
\begin{lemma}[Adjoint of $\cL_0$]\label[lemma]{S3:lm0}
    The $L^2$-adjoint of $\mathcal{L}_{0}$ is
\begin{equation*}
\mathcal{L}_{0}^{*}\coloneqq\ka_0^2 \frac{d^2}{dz^2} + \ka_0 (\mathcal{M}_{\ka_0} - c_0) \frac{d}{dz} + r\cI, \end{equation*} 
where $\cL_{0}^{*}:X\to Y$ with $X$ and $Y$ be as defined in \eqref{S3:eq8} and \begin{equation}\label{S3:eqm2}
\ker(\cL_{0}^{*})=\operatorname{span} \left\{\cos z, \sin z\right\}.
\end{equation}
\end{lemma}
\begin{proof}
 We know that there exists a unique $L^2$-adjoint of the operator $\mathcal{L}_{0}$ satisfying
\begin{equation*}
    \langle \cL_0 u, v \rangle_{L^2_{\text{per}}(\mathbb{T})}= \langle u, \cL_0^* v \rangle_{L^2_{\text{per}}(\mathbb{T})} \quad \text{for all } u,v \in C^\infty_\text{per}(\mathbb{T}).
\end{equation*}
The $L^2$-adjoint of $\frac{d}{dz}$ is $-\frac{d}{dz}$ and $\cM_\ka$ is a self-adjoint operator. Then using \eqref{S3:eq15} we obtain
\begin{equation*}
\mathcal{L}_{0}^{*}=\ka_0^2 \frac{d^2}{dz^2} + \ka_0 (\mathcal{M}_{\ka_0} - c_0) \frac{d}{dz} + r\cI, \end{equation*}
which can be extended to be defined from $X$ to $Y$ by using the pseudodifferential order $\mu=\max\{2,\beta+1\}$ of $\mathcal{L}_{0}^{*}$ and the fact that $C^\infty_\text{per}(\mathbb{T})$ is dense in $X$. To calculate the kernel of $\mathcal{L}_{0}^{*}$ we set
\begin{align*}
\cL_{0}^{*}e^{inz} &= \left( -\ka_{0}^2 n^2 + \ka_{0}(m(\ka_{0} n) - c_{0})in + r \right)e^{inz} = 0.
\end{align*}
Comparing the real and imaginary parts, we have
\begin{align*}
r - \ka_{0}^2 n^2 = 0 \implies r(1-n^2)=0 \implies n=\pm 1,& \\
\text{and } \ka_{0}(m(\ka_{0} n) - c_{0})n = 0 \text{ which is true for } n=\pm 1.&
\end{align*}
Again, taking a $g(z)\in X$ and using its Fourier series, we can easily see that
\begin{equation*}
\ker(\cL_{0}^{*}) = \operatorname{span}\{e^{-iz},e^{iz}\}=\operatorname{span} \left\{\cos z, \sin z\right\}.
\end{equation*}
\end{proof}
\begin{lemma}[Fredholm Operator]\label[lemma]{S3:lm1}
    $\cL_{0}$ is a Fredholm operator of index zero.
\end{lemma}
\begin{proof}
    Now from \eqref{S3:eq14}, we have 
\begin{equation*}
\ker(\cL_0)=\operatorname{span}\left\{e^{-iz},e^{iz}\right\}=\operatorname{span}\left\{\cos{z},\sin{z}\right\}.
\end{equation*}
Therefore, $\dim\ker(\cL_{0})=2$, and hence $\ker(\cL_{0})$ is finite-dimensional. We shall prove that 
\begin{equation*}
\operatorname{range}(\cL_0) = \left\{ f \in Y \bigm| \widehat{f}(-1) = \widehat{f}(1) = 0 \right\}.
\end{equation*}
Let $f\in\operatorname{range}(\cL_{0})$, then there exists $w\in X$ such that $\cL_{0}w=f$. Using the Fourier series expansion of $w$ and $f$ and comparing the Fourier coefficients for each $n\in\mathbb{Z}$, we get that $\widehat{f}(-1) = \widehat{f}(1) = 0$. Conversely, suppose that $f\in Y$ and $\widehat{f}(-1) = \widehat{f}(1) = 0$. We define $w$ for which $$\widehat{w}(n) = \frac{\widehat{f}(n)}{-\ka_0^2 n^2 - \ka_0(m(\ka_0 n) - c_0)in + r}$$ for $n\neq\pm 1$ and $\widehat{w}(-1) = \widehat{w}(1) = 0$. It can be shown that $\cL_{0}w=f$. Now we prove that such $w\in X$. 
\begin{align*}
    \|w\|^{2}_{X}=\sum\limits_{n\in\Z}\left(1+|n|^{2}\right)^{s+\mu}|\widehat{w}(n)|^{2}
    =\sum\limits_{n\in\Z\setminus\{\pm1\}}\left(1+|n|^{2}\right)^{s+\mu}\frac{|\widehat{f}(n)|^{2}}{(-\ka_0^2 n^2+r)^{2}+( \ka_0n(m(\ka_0 n) - c_0))^{2}}.
\end{align*}
Let us define 
\begin{align*}
    g(n)&=(-\ka_0^2 n^2+r)^{2}+( \ka_0n(m(\ka_0 n) - c_0))^{2}\\ 
    &\geq \max \left\{(-\ka_0^2 n^2+r)^{2},( \ka_0n(m(\ka_0 n) - c_0))^{2}\right\}\\
    &\geq \max\left\{C_1 |n|^{4}, C_2 |n|^{2\beta+2}\right\} \qquad & \text{for } |n|\geq N \text{ and positive constants } C_1, C_2, N\\
    &\geq C_3(1+|n|^{2})^{\mu}  \qquad & \text{for } |n|\geq N \text{ and positive constant } C_3. 
\end{align*}
We set $$C_4=\min_{\substack{0<|n|<N\\|n|\neq 1}}\left\{\frac{g(n)}{(1+|n|^{2})^{\mu}}\right\}.$$ Then $g(n)\geq C_5 (1+|n|^{2})^{\mu}$ where $C_5=\min\{C_3,C_4\}$. Therefore, $\|w\|_{X}\leq C\|f\|_{Y}<\infty$.

Next we define the linear functional $$\phi_n : Y \to \mathbb{R}$$ by $$\phi_n(f) = \widehat{f}(n).$$
It can be shown that for any $f\in Y=H^{s}_{\text{per}}(\mathbb{T})$, $s>1/2$,
\begin{equation*}
|\phi_n(f)|^{2} = |\widehat{f}(n)|^{2}\leq \frac{1}{(1+|n|^{2})^{s}}\|f\|_Y^2.
\end{equation*}
Therefore, 
\begin{equation*}
|\phi_n(f)|\leq M \|f\|_Y, \text{ where } M=\frac{1}{(1+|n|^{2})^{\frac{s}{2}}} > 0.
\end{equation*}
Thus, $\phi_n$ is bounded and therefore continuous. This implies that the maps $\phi_{-1} : Y \to \mathbb{R}$ and $\phi_{1} : Y \to \mathbb{R}$ are continuous, and hence the following sets are closed
\begin{equation*}
S_1 = \{ f \in Y \mid \widehat{f}(-1) = 0 \}, \quad
S_2 = \{ f \in Y \mid \widehat{f}(1) = 0 \}.
\end{equation*}
Therefore, $\operatorname{range}(\cL_0) = S_1 \cap S_2$ is closed. Also,
\begin{equation*}
\ker(\cL_0^{*}) = (\operatorname{range}(\cL_{0}))^\perp \cong Y/\operatorname{range}(\cL_{0})=\operatorname{cokernel}(\cL_0).
\end{equation*}
Therefore, by \eqref{S3:eqm2}, $\dim\operatorname{cokernel}(\cL_{0})=2$, and hence $\operatorname{cokernel}(\cL_{0})$ is finite-dimensional. Thus, $\cL_0$ is a Fredholm operator with index $0$.
\end{proof}

\subsubsection{Lyapunov--Schmidt Reduction}
We now prove the existence of periodic traveling waves bifurcating from $(0, \ka_0, c_0)$. We have $F(0,\ka_0,c_0)=0$. We can define the $L^2$-orthogonal projection $\Pi: Y \to \ker(\cL_0)$ by
\begin{equation}\label{S3:eq19}
(\Pi f)(z) = \frac{1}{\pi} \left( \int_{0}^{2\pi} f(s) \cos(s) ds \right) \cos z + \frac{1}{\pi} \left( \int_{0}^{2\pi} f(s) \sin(s) ds \right) \sin z,
\end{equation}
and $(\cI-\Pi):Y\to(\cI-\Pi)Y$ by
\begin{equation}\label{S3:eq20}
    ((\cI-\Pi)f)(z)=f(z)-(\Pi f)(z).
\end{equation}
Since $\ker(\cL_0)$ and $\operatorname{range}(\cL_0)$ are closed we can split our function spaces $X = X_1 \oplus X_2$ and $Y = Y_1 \oplus Y_2$ dictated by the projections defined in \eqref{S3:eq19} and \eqref{S3:eq20} as
\begin{align*}
X_1 &= \ker(\cL_0) \subset X, \quad &X_2 = (\cI - \Pi)X, \\
Y_1 &= \ker(\cL_0) \subset Y, \quad &Y_2 = (\cI - \Pi)Y.
\end{align*}
So, any solution $w\in X$ can be decomposed as $w(z)=a\cos{z}+b\sin{z}+v(z)$ where $v\in X_2$. We can express this in the form 
\begin{equation*}
    w(z)=\eps \cos{(z-\theta)}+v(z),
\end{equation*} such that $\eps = \sqrt{a^2+b^2}$ and $\theta$ is such that $\cos{\theta}=\frac{a}{a^{2}+b^{2}}$ and $\sin{\theta}=\frac{b}{a^{2}+b^{2}}$. Due to the translation invariance of the PDE \eqref{S1:eq1} (if $w(z)$ is a solution, so is $w(z+\theta)$), we can take the phase angle $\theta=0$. Therefore, we decompose the solution $w \in X$ as
\begin{equation}\label{S3:eq21}
w = \eps \cos z + v,
\end{equation}
where $\eps \in \R$ is an amplitude parameter and $v \in X_2$ contains all higher-order harmonics ($v$ is $L^2$-orthogonal to both $\cos z$ and $\sin z$). Also, since $\ker(\cL(\ka_0,c_0))$ is nontrivial, to look for a nontrivial solution in a neighborhood of $\ka_0$ and $c_0$, we introduce parameter shifts $\tilde{\ka}\in\R$ and $\tilde{c}\in\R$ such that
\begin{equation}\label{S3:eq22}
\ka = \ka_0 + \tilde{\ka} \quad \text{and} \quad c = c_0 + \tilde{c}.
\end{equation}
Substituting \eqref{S3:eq21} and \eqref{S3:eq22} into \eqref{S3:eq9} and applying the projections defined in \eqref{S3:eq19} and \eqref{S3:eq20}, we split the infinite-dimensional problem \eqref{S3:eq9} into a coupled system defined as
\begin{empheq}
[left=\empheqlbrace]{align}(\cI - \Pi) F(\eps \cos z + v, \ka_0 + \tilde{\ka}, c_0 + \tilde{c}) &= 0, \quad \text{(The Range Equation)} \label{S3:eq23} \\
\Pi F(\eps \cos z + v, \ka_0 + \tilde{\ka}, c_0 + \tilde{c}) &= 0. \quad \text{(The Bifurcation Equation)} \label{S3:eq24}
\end{empheq}
We start with the Range Equation \eqref{S3:eq23} and define $G: \R \times X_2 \times \R \times \R \to Y_2$ as 
\begin{equation*}
    G(\eps, v, \tilde{\ka}, \tilde{c}) = (\cI - \Pi) F(\eps \cos z + v, \ka_0 + \tilde{\ka}, c_0 + \tilde{c}).
\end{equation*}
Since $w\equiv 0$ is a solution of \eqref{S3:eq9} for all $\ka$ and $c$, $G(0,0,0,0) = (\cI-\Pi)F(0,\ka_0,c_0) = 0$. $G$ is analytic because it is the composition of a bounded linear projection $(\cI-\Pi)$ and the analytic map $F$ as in \cref{S3:analytic}. The Fr\'echet derivative of $G$ with respect to $v$ at a point $(\eps,v,\tilde{\ka},\tilde{c})$ in the direction $\tilde{v}\in X_2$ is
\begin{equation*}
   \partial_v G(\eps,v,\tilde{\ka},\tilde{c})\tilde{v} = (\cI - \Pi) \partial_w F(\eps \cos z + v, \ka_0 + \tilde{\ka}, c_0 + \tilde{c})\tilde{v}.
\end{equation*}
Therefore, using the fact that $\ker(\Pi)=\operatorname{range}(\cL_0)$, we have
\begin{equation*}
\partial_v G(0,0,0,0) = (\cI - \Pi) \partial_w F(0, \ka_0, c_0) |_{X_2} = \cL_0 |_{X_2}.
\end{equation*}
The operator $\cL_0$ restricts to a bounded isomorphism mapping from $X_2$ to $Y_2$ as the functions that correspond to a nontrivial kernel have been removed from the domain of $\cL_0$. Also, the functions with nonzero Fourier coefficients corresponding to $n=\pm 1$ have been removed from the co-domain of $\cL_0$. Therefore, $\partial_v G(0,0,0,0):X_2\to Y_2$ is a bounded linear isomorphism. By the Implicit Function Theorem (\cref{S2:thm1}), there exists a unique analytic map $v^*:\R\times\R\times\R\to X_2$ such that $v^*(\eps, \tilde{\ka}, \tilde{c}) \in X_2$, defined in a neighborhood of $(0,0,0)$ such that $v^*(0,0,0) = 0$ and \begin{equation*}
    G(\eps,v^*(\eps,\tilde{\ka},\tilde{c}),\tilde{\ka},\tilde{c})\equiv 0.
\end{equation*}
Therefore, 
\begin{equation}\label{S3:eq27}
(\cI - \Pi) F(\eps \cos z + v^*(\eps, \tilde{\ka}, \tilde{c}), \ka_0 + \tilde{\ka}, c_0 + \tilde{c}) \equiv 0.
\end{equation}
Because $F(0,\ka,c) = 0$ regardless of the value of the parameters $\ka$ and $c$, uniqueness dictates that $v^*(0, \tilde{\ka}, \tilde{c}) = 0$ in a neighborhood of $(0,0,0)$. By Taylor expansion of $v^*$, we conclude that $v^* = \mathcal{O}(\eps)$.
\vspace{10pt}
\begin{lemma}[Derivatives of $v^*$]\label[lemma]{S3:lm2}
The unique solution to the Range Equation, $v^*(\eps, \tilde{\ka}, \tilde{c}) \in X_2$, satisfies
\begin{equation}\label{S3:eq28}
\partial_\eps v^*(0, 0, 0) = 0.
\end{equation}
Consequently, for fixed parameters $\ka=\ka_0$ and $c=c_0$, the Taylor expansion of $v^*$ with respect to the amplitude $\eps$ contains no linear terms, implying $v^* = \mathcal{O}(\eps^2)$. Also, \begin{equation}\label{S3:eq29}
\partial_{\tilde{\ka}} v^*(0, 0, 0) = 0,
\end{equation}
and \begin{equation}\label{S3:eq30}
\partial_{\tilde{c}} v^*(0, 0, 0) = 0.
\end{equation}
\end{lemma}
\begin{proof}
Differentiating \eqref{S3:eq27} with respect to $\eps$, we get
\begin{equation*}
   (\cI - \Pi) \partial_{w} F(\eps \cos{z} + v^*(\eps, \tilde{\ka}, \tilde{c}), \ka_0 + \tilde{\ka}, c_0 + \tilde{c})\left[\cos{z}+\partial_{\eps}v^*(\eps,\tilde{\ka},\tilde{c})\right]= 0,
\end{equation*}
and evaluating it at the origin $(\eps, \tilde{\ka}, \tilde{c}) = (0, 0, 0)$ yields
\begin{equation*}
(\cI - \Pi) \partial_w F(0, \ka_0, c_0) \left[ \cos z + \partial_\eps v^*(0, 0, 0) \right] = 0.
\end{equation*}
Therefore, by linearity,
\begin{equation*}
(\cI - \Pi) \cL_0 (\cos z) + (\cI - \Pi) \cL_0 (\partial_\eps v^*(0, 0, 0)) = 0.
\end{equation*}
Because $\cos z \in \ker(\cL_0)$, we have $\cL_0 (\cos z) = 0$. Furthermore, since $\ker(\Pi)=\operatorname{range}(\cL_0)$ we have $(\cI-\Pi) \cL_0 = \cL_0$. Thus, the equation simplifies to
\begin{equation*}
\cL_0 \left( \partial_\eps v^*(0, 0, 0) \right) = 0.
\end{equation*}
However, $\partial_\eps v^*(0, 0, 0) \in X_2$ by construction, and $X_2 \cap \ker(\cL_0) = \{0\}$ by the direct sum decomposition, forcing $\partial_\eps v^*(0, 0, 0) = 0$. Since, we previously established $v^*(0,0,0) = 0$ for fixed parameters $\ka=\ka_0$ and $c=c_0$, the Taylor expansion of $v^*$ with respect to the amplitude $\eps$ contains no linear terms, implying $v^* = \mathcal{O}(\eps^2)$.

Differentiating \eqref{S3:eq7} with respect to $\ka$ evaluated at a point $(w,\ka,c)$ is given by
\begin{equation*}
\partial_{\ka}F(w,\ka,c)=2\ka w''-\ka\cM'_{\ka}w'-\cM_{\ka}w'+cw'-\alpha ww',
\end{equation*}
which when evaluated at the point $(0,\ka_0,c_0)$ yields
\begin{equation}\label{S3:eq102}
\partial_{\ka}F(0,\ka_0,c_0)=0.
\end{equation}
Now, differentiating \eqref{S3:eq27} with respect to $\tilde{\ka}$, we get
\begin{align*}
    &(\cI - \Pi) \partial_{\ka} F(\eps \cos{z} + v^*(\eps, \tilde{\ka}, \tilde{c}), \ka_0 + \tilde{\ka}, c_0 + \tilde{c})\\&+(\cI - \Pi) \partial_{w} F(\eps \cos{z} + v^*(\eps, \tilde{\ka}, \tilde{c}), \ka_0 + \tilde{\ka}, c_0 + \tilde{c})\left[\partial_{\tilde{\ka}}v^*(\eps,\tilde{\ka},\tilde{c})\right] = 0,
\end{align*}
and evaluating it at the origin $(\eps, \tilde{\ka}, \tilde{c}) = (0, 0, 0)$ yields 
\begin{equation*}
   (\cI - \Pi) \partial_{\ka}F(0,\ka_0, c_0)+(\cI - \Pi) \partial_{w}F(0,\ka_0, c_0)\left[\partial_{\tilde{\ka}}v^*(0,0,0)\right] = 0. 
\end{equation*}
Therefore, by \eqref{S3:eq102} and the fact that $\ker(\Pi)=\operatorname{range}(\cL_0)$, we get
\begin{equation*}
\cL_0(\partial_{\tilde{\ka}}v^*(0,0,0))=0,
\end{equation*}
which by $X_2 \cap \ker(\cL_0) = \{0\}$ forces $\partial_{\tilde{\ka}}v^*(0,0,0)=0$. 
Following similar arguments we can prove that $\partial_{\tilde{c}} v^*(0, 0, 0) = 0$.
\end{proof}
To solve the finite-dimensional Bifurcation Equation we first substitute the mapping $v^*$ into \eqref{S3:eq24} to obtain a reduced equation $\Phi(\eps, \tilde{\ka}, \tilde{c}) = 0$, where $\Phi: \R \times \R \times \R \to \ker(\cL_0) \cong \R^2$ and is defined by
\begin{equation*}
\Phi(\eps, \tilde{\ka}, \tilde{c}) \coloneqq \Pi F(\eps \cos z + v^*(\eps, \tilde{\ka}, \tilde{c}), \ka_0 + \tilde{\ka}, c_0 + \tilde{c}).
\end{equation*}
Since $\Phi(0, \tilde{\ka}, \tilde{c}) = 0$ for all $\tilde{\ka},\tilde{c}\in\R$, the Jacobian matrix of $\Phi$ with respect to $(\tilde{\ka},\tilde{c})$ evaluated at the point $(0,0,0)$, $\partial_{(\tilde{\ka},\tilde{c})}\Phi(0,0,0)$, is the singular zero matrix. Therefore, the Implicit Function Theorem (\cref{S2:thm1}) does not apply to $\Phi=0$ at the origin. Instead, we factor out the $\eps$ parameter and define the divided difference function
\begin{equation*}
\tilde{\Phi}(\eps, \tilde{\ka}, \tilde{c}) \coloneqq 
\begin{cases} 
\frac{1}{\eps} \Phi(\eps, \tilde{\ka}, \tilde{c}), & \eps \neq 0, \\
\partial_\eps \Phi(0, \tilde{\ka}, \tilde{c}), & \eps = 0.
\end{cases}
\end{equation*}
We aim to apply the Implicit Function Theorem (\cref{S2:thm1}) to $\tilde{\Phi} = 0$ at the origin $(0,0,0)$ to solve for $\tilde{\ka}$ and $\tilde{c}$ as functions of $\eps$. We treat $\tilde{\Phi} =(\tilde{\Phi}_1, \tilde{\Phi}_2)^T$ as a vector corresponding to the coefficients of $\cos z$ and $\sin z$. The Fr\'echet derivative of $\Phi$ with respect to $\eps$ evaluated at a point $(\eps,\tilde{\ka},\tilde{c})$ is
\begin{equation}\label{S3:eq33}
\partial_{\eps}\Phi(\eps,\tilde{\ka},\tilde{c})=\Pi \partial_{w} F(\eps \cos{z} + v^*(\eps, \tilde{\ka}, \tilde{c}), \ka_0 + \tilde{\ka}, c_0 + \tilde{c})\left[\cos{z}+\partial_{\eps}v^*(\eps,\tilde{\ka},\tilde{c})\right].
\end{equation}
Evaluating $\tilde{\Phi}$ at the base point gives
\begin{equation*}
\tilde{\Phi}(0,0,0) = \partial_\eps \Phi(0,0,0) = \Pi \left( \partial_w F(0,\ka_0,c_0) [ \cos z + \partial_\eps v^*(0,0,0) ] \right).
\end{equation*}
Because $\partial_w F(0,\ka_0,c_0) = \cL_0$, and $\Pi \cL_0 = 0$, we have $\tilde{\Phi}(0,0,0) = 0$. By definition of $\tilde{\Phi}$, it is an analytic map. Differentiating \eqref{S3:eq33} with respect to $\tilde{\ka}$ we obtain
\begin{align*}
\partial_{\tilde{\ka}}\partial_{\eps}\Phi(\eps,\tilde{\ka},\tilde{c})&=
    \Pi \partial_{w} F(\eps \cos{z} + v^*(\eps, \tilde{\ka}, \tilde{c}), \ka_0 + \tilde{\ka}, c_0 + \tilde{c})\partial_{\tilde{\ka}}\partial_{\eps}v^*(\eps,\tilde{\ka},\tilde{c})\\&\quad +\Pi\partial_{\ka} \partial_{w} F(\eps \cos{z} + v^*(\eps, \tilde{\ka}, \tilde{c}), \ka_0 + \tilde{\ka}, c_0 + \tilde{c})\left[\cos{z}+\partial_{\eps}v^*(\eps,\tilde{\ka},\tilde{c})\right]\\&\quad+\Pi \partial_{ww} F(\eps \cos{z} + v^*(\eps, \tilde{\ka}, \tilde{c}), \ka_0 + \tilde{\ka}, c_0 + \tilde{c})\partial_{\tilde{\ka}}v^*(\eps,\tilde{\ka},\tilde{c})\left[\cos{z}+\partial_{\eps}v^*(\eps,\tilde{\ka},\tilde{c})\right].    
\end{align*}
Therefore,
\begin{align*}
\partial_{\tilde{\ka}}\partial_{\eps}\Phi(0,0,0)&=
    \Pi \partial_{w} F(0, \ka_0, c_0)\partial_{\tilde{\ka}}\partial_{\eps}v^*(0,0,0)\\&\quad+\Pi\partial_{\ka} \partial_{w} F(0, \ka_0, c_0)\left[\cos{z}+\partial_{\eps}v^*(0,0,0)\right]\\&\quad+\Pi \partial_{ww} F(0, \ka_0, c_0)\partial_{\tilde{\ka}}v^*(0,0,0)\left[\cos{z}+\partial_{\eps}v^*(0,0,0)\right].
\end{align*}
Using the fact that $\Pi\cL_0=0$, \eqref{S3:eq28}, and \eqref{S3:eq29} we find
\begin{equation*}
\partial_{\tilde{\ka}} \tilde{\Phi}(0,0,0) = \partial_{\tilde{\ka}} \partial_\eps \Phi(0,0,0) = \Pi \partial_{\ka} \left( \partial_w F(0, \ka_0, c_0) [\cos z] \right).
\end{equation*}
Using \eqref{S3:eq12} yields
\begin{equation}\label{S3:eq36}
\partial_w F(0, \ka, c) \cos z = -\ka^2 \cos z + \ka(m(\ka)-c)\sin z + r \cos z.
\end{equation}
Differentiating this with respect to $\ka$ and evaluating at $(\ka_0, c_0)$ yields
\begin{align*}
\left. \partial_\ka \left( \partial_w F(0, \ka, c) \cos z \right) \right|_{(\ka_0, c_0)} &= -2\ka_0 \cos z + \left[ (m(\ka_0)-c_0) + \ka_0 m'(\ka_0) \right]\sin z \\
&= -2\ka_0 \cos z + \ka_0 m'(\ka_0) \sin z.
\end{align*}
Projecting this onto $\cos z$ and $\sin z$ gives
\begin{equation}\label{S3:eq37}
\partial_{\tilde{\ka}} \tilde{\Phi}_1(0,0,0) = -2\ka_0 \quad \text{and} \quad \partial_{\tilde{\ka}} \tilde{\Phi}_2(0,0,0) = \ka_0 m'(\ka_0).
\end{equation}
Similarly, using \eqref{S3:eq28} and \eqref{S3:eq30} we can obtain 
\begin{equation*}
\partial_{\tilde{c}} \tilde{\Phi}(0,0,0) = \partial_{\tilde{c}} \partial_\eps \Phi(0,0,0) = \Pi \partial_{c} \left( \partial_w F(0, \ka_0, c_0) [\cos z] \right).
\end{equation*}
Differentiating \eqref{S3:eq36} with respect to $c$ and evaluating at $(\ka_0,c_0)$ yields
\begin{equation*}
\left. \partial_c \left( \partial_w F(0, \ka, c) \cos z \right) \right|_{(\ka_0, c_0)} = -\ka_0 \sin z.
\end{equation*}
Projecting onto $\cos{z}$ and $\sin{z}$ gives
\begin{equation}\label{S3:eq39}
\partial_{\tilde{c}} \tilde{\Phi}_1(0,0,0) = 0 \quad \text{and} \quad \partial_{\tilde{c}} \tilde{\Phi}_2(0,0,0) = -\ka_0.
\end{equation}
Therefore, using \eqref{S3:eq37} and \eqref{S3:eq39} the Jacobian matrix evaluated at the origin is
\begin{equation*}
J = \partial_{(\tilde{\ka},\tilde{c})}\tilde{\Phi}(0,0,0)=\begin{pmatrix} 
-2\ka_0 & 0 \\
\ka_0 m'(\ka_0) & -\ka_0
\end{pmatrix}.
\end{equation*}
Now $\det(J) = (-2\ka_0)(-\ka_0) = 2\ka_0^2= 2r > 0$. The Jacobian is nonsingular. Therefore, by the Implicit Function Theorem (\cref{S2:thm1}), there exist unique analytic mappings $\tilde{\ka}(\eps)$ and $\tilde{c}(\eps)$ defined for $\eps$ sufficiently small, defined in a neighborhood of $0$ such that $\tilde{\ka}(0)=0$, $\tilde{c}(0)=0$, and $\tilde{\Phi}(\eps, \tilde{\ka}(\eps), \tilde{c}(\eps)) \equiv 0$. Thus, we have proved the existence of an analytic, continuous branch of nontrivial periodic traveling wave solutions parameterized by the small amplitude $\eps$.

\subsection{Asymptotic Expansions}
Having proven existence, we now derive the explicit asymptotic expansions of the solutions. We expand the wave profile, wavenumber, and wave speed in powers of $\eps$ as 
\begin{align*}
w(z) &= \eps w_1(z) + \eps^2 w_2(z) + \eps^3 w_3(z) + \mathcal{O}(\eps^4), \\
\ka &= \ka_0 + \eps \ka_1 + \eps^2 \ka_2 + \mathcal{O}(\eps^3), \\
c &= c_0 + \eps c_1 + \eps^2 c_2 + \mathcal{O}(\eps^3). 
\end{align*}
Because of \eqref{S3:eq21}, $w_1(z) = \cos z$ and $w_j(z) \perp \cos z, \sin z$ for $j \geq 2$. We substitute these expansions into $F(w, k, c) = 0$ and match coefficients at each order of $\eps$. We get
\begin{align*}
F(w(z),\ka,c)&=\left(\ka_0+\eps\ka_1+\eps^{2}\ka_2\right)^{2}\left(\eps w_1''(z) + \eps^2 w_2''(z) + \eps^3 w_3''(z)\right)\\ &\quad-\left(\ka_0+\eps\ka_1+\eps^{2}\ka_2\right)\left(\cM_{\ka_0+\eps\ka_1+\eps^{2}\ka_2}-c_0-\eps c_1-\eps^{2}c_2\right)\left(\eps w_1'(z) + \eps^2 w_2'(z) + \eps^3 w_3'(z)\right)\\&\quad-\alpha\left(\ka_0+\eps\ka_1+\eps^{2}\ka_2\right)\left(\eps w_1(z) + \eps^2 w_2(z) + \eps^3 w_3(z)\right)\left(\eps w_1'(z) + \eps^2 w_2'(z) + \eps^3 w_3'(z)\right)\\&\quad+r\left(\eps w_1(z) + \eps^2 w_2(z) + \eps^3 w_3(z)\right)\left(1-\eps w_1(z) - \eps^2 w_2(z) - \eps^3 w_3(z)\right)+\cdots\\&= \,0.
\end{align*}
Using Assumption \ref{assumption1} on the symbol $m$, we get
\begin{align*}
&\cM_{\ka_0+\eps\ka_1+\eps^{2}\ka_2}\left(e^{inz}\right)\\&= m\left(\left(\ka_0+\eps\ka_1+\eps^{2}\ka_2\right)n\right)e^{inz}\\&=\left(m(\ka_{0}n)+\left(\eps\ka_1+\eps^{2}\ka_2\right)nm'(\ka_{0}n)+\frac{\eps^2\ka_1^2n^{2}}{2}m''(\ka_{0}n)+\mathcal{O}(\eps^{3})\right)e^{inz}\\&=\cM_{\ka_0}e^{inz}+\left(\eps \ka_{1}\right)\cM'_{\ka_0}e^{inz}+\eps^{2}\left(\ka_{2}\cM'_{\ka_0}+\frac{\ka_{1}^{2}}{2}\cM''_{\ka_0}\right)e^{inz}+\mathcal{O}(\eps^{3}),
\end{align*}
where
\begin{align*}
\cM_{\ka_0}e^{inz}&=m(\ka_{0}n)e^{inz},\\\cM'_{\ka_0}e^{inz}&=nm'(\ka_{0}n)e^{inz},\\\cM''_{\ka_0}e^{inz}&=n^{2}m''(\ka_{0}n)e^{inz}.
\end{align*}

\subsubsection{Order \texorpdfstring{$\mathcal{O}(\eps)$}{O eps}}
The $\mathcal{O}(\eps)$ equation is simply the linearized equation evaluated at the bifurcation point
\begin{equation}\label{S3:eq44}
\cL_0 w_1 = \ka_0^2 w_1'' - \ka_0(\cM_{\ka_0} - c_0)w_1' + r w_1 = 0.
\end{equation}
By our choice of $\ka_0=\sqrt{r}$ and $c_0=m(\sqrt{r})$, we see that $w_1(z) = \cos z$ satisfies \eqref{S3:eq44} and hence is a solution. 

\subsubsection{Order \texorpdfstring{$\mathcal{O}(\eps^2)$}{O eps 2}}
Collecting all terms of order $\eps^2$ and equating to $0$ we obtain
\begin{align}\label{S3:eq103}
&\ka_{0}^{2}w''_2+2\ka_{0}\ka_{1}w''_{1}-\ka_{0}\ka_{1}\cM'_{\ka_0}w'_{1}-\ka_{1}\cM_{\ka_{0}}w'_{1}-\ka_{0}\cM_{\ka_{0}}w'_{2}+\ka_{0}c_{0}w'_{2}\notag\\+&\ka_{0}c_{1}w'_{1}+\ka_{1}c_{0}w'_{1}-\alpha\ka_{0}w_1w'_1+rw_2-rw_{1}^{2}=0.
\end{align}
Using the expression for $\cL_{0}$ from \eqref{S3:eq15} and the fact that $c_{0}=m(\ka_{0})$ and $w_{1}(z)=\cos{z}$, the above equation \eqref{S3:eq103} can be rearranged as
\begin{equation} \label{S3:eq47}
\cL_0 w_2 = (2 \ka_0\ka_1) \cos z + (\ka_0 c_{1}-\ka_0 \ka_1 m'(\ka_0)) \sin z + \frac{r}{2} + \frac{r}{2} \cos(2z) - \frac{\alpha \ka_0}{2} \sin(2z).
\end{equation}
By the Fredholm Alternative \citep[see][\S 2.2.5]{Kapitula2013}, for a periodic solution $w_2(z)$ to exist, the right-hand side of \eqref{S3:eq47} must be orthogonal to $\ker(\cL_0^*) =\operatorname{span}\{\cos z, \sin z\}$. Thus, the coefficients of $\cos z$ and $\sin z$ in \eqref{S3:eq47} must be zero. Hence $\ka_{1}=0$ and $c_{1}=0$, which proves that the first-order corrections to the wavenumber and wave speed vanish. The equation for $w_2$ simplifies to
\begin{equation}\label{S3:eq48}
\cL_0 w_2 = \frac{r}{2} + \frac{r}{2} \cos(2z) - \frac{\alpha \ka_0}{2} \sin(2z).
\end{equation}
We seek a solution of the form $w_2(z) = A_0 + A_2 \cos(2z) + B_2 \sin(2z)$ to \eqref{S3:eq48}. 
Applying $\cL_0$ we have
\begin{align*}
\cL_0(A_0) &= r A_0, \\
\cL_0(A_2 \cos(2z) + B_2 \sin(2z)) &= A_2 \left[ -4\ka_0^2 \cos(2z) + 2\ka_0(m(2\ka_0)-c_0)\sin(2z) + r \cos(2z) \right] \\
&\quad + B_2 \left[ -4\ka_0^2 \sin(2z) - 2\ka_0(m(2\ka_0)-c_0)\cos(2z) + r \sin(2z) \right].
\end{align*}
Comparing coefficients,
\begin{align}
r A_0 = \frac{r}{2} \implies A_0 &= \frac{1}{2},\notag \\
-3r A_2 - 2\ka_0 (m(2\ka_0)-c_0) B_2 &= \frac{r}{2},\label{S3:eq50} \\
2\ka_0 (m(2\ka_0)-c_0) A_2 - 3r B_2 &= -\frac{\alpha \ka_0}{2}.\label{S3:eq51}
\end{align}
Solving for $A_2$ and $B_2$ from \eqref{S3:eq50} and \eqref{S3:eq51} gives
\begin{equation*}
A_2  = -\frac{3r^2 + 2 \alpha \ka_0^2 (m(2\ka_0)-c_0)}{2(9r^2 + 4\ka_0^2 (m(2\ka_0)-c_0)^2)},\quad
B_2  = \frac{3 \alpha r \ka_0 - 2 r \ka_0 (m(2\ka_0)-c_0)}{2(9r^2 + 4\ka_0^2 (m(2\ka_0)-c_0)^2)}. 
\end{equation*}
Therefore,
\begin{equation*}
w_2(z)=\frac{1}{2}-\left(\frac{3r^2 + 2 \alpha \ka_0^2 (m(2\ka_0)-c_0)}{2(9r^2 + 4\ka_0^2 (m(2\ka_0)-c_0)^2)}\right)\cos{2z}+\left(\frac{3 \alpha r \ka_0 - 2 r \ka_0 (m(2\ka_0)-c_0)}{2(9r^2 + 4\ka_0^2 (m(2\ka_0)-c_0)^2)}\right)\sin{2z}.
\end{equation*}

\subsubsection{Order \texorpdfstring{$\mathcal{O}(\eps^3)$}{O eps 3}}
Collecting all terms of order $\eps^3$ and equating to $0$ and using that $\ka_1=c_1=0$, and rearranging using \eqref{S3:eq15}, we obtain
\begin{equation}\label{S3:eq55}
\cL_{0}w_3=2\ka_{0}\ka_2\cos{z}+(\ka_0c_2 -\ka_0\ka_2m'(\ka_0))\sin{z}+2rw_1w_2+\alpha\ka_0(w_1w_2)'.
\end{equation}
We compute $w_1 w_2$ as
\begin{align}\label{S3:eq56}
w_1 w_2 &= \cos z \left( \frac{1}{2} + A_2 \cos(2z) + B_2 \sin(2z) \right) \nonumber \\
&= \left( \frac{1+A_2}{2} \right) \cos z + \frac{B_2}{2} \sin z + \frac{A_2}{2}\cos{3z}+\frac{B_2}{2}\sin{3z}.
\end{align}
Taking the derivative, we have
\begin{equation}\label{S3:eq57}
(w_1 w_2)' = - \left( \frac{1+A_2}{2} \right) \sin z + \frac{B_2}{2} \cos z -\frac{3}{2}A_{2}\sin{3z}+\frac{3}{2}B_2\cos{3z}.
\end{equation}
Substituting \eqref{S3:eq56} and \eqref{S3:eq57} in \eqref{S3:eq55}, and applying the Fredholm Alternative again, we can equate the coefficients of $\cos{z}$ and $\sin{z}$ to $0$ to obtain
\begin{align*}
2\ka_0 \ka_2 + r(1+A_2) + \frac{\alpha \ka_0 B_2}{2} &= 0, \\
-\ka_0 \ka_2m'(\ka_0) +\ka_0 c_2 + r B_2 - \frac{\alpha \ka_0(1+A_2)}{2} &= 0.
\end{align*}
Solving for $\ka_2$ and $c_2$ we get
\begin{align*}
\ka_2 &= -\frac{r(1+A_2) + \frac{1}{2}\alpha \ka_0 B_2}{2\ka_0},  \quad
c_2 = \frac{(\alpha\ka_0-rm'(\ka_0) )(1+A_2) - (\frac{1}{2}\alpha \ka_0m'(\ka_0) +2r) B_2}{2\ka_0}. 
\end{align*}

\section{The Spectral Stability Problem}\label{section4}
In \S\ref{section3} we found a family of $2\pi$-periodic traveling waves for fixed $r>0$ and $\alpha\in\R\setminus\{0\}$ given by
\begin{equation}\label{S4:eq1}
u(x,t) = w^{\eps}\left(\ka(\eps)\left(x - c(\eps)t\right)\right) =: w^{\eps}(z).
\end{equation}
We attempt to study the spectral stability of \eqref{S4:eq1} as solutions to the original PDE \eqref{S1:eq1}.

\subsection{The Linearized Spectral Problem}
We fix the value of $\eps$ and consider a unique $2\pi$-periodic traveling wave solution to \eqref{S1:eq1} of the form \eqref{S4:eq1} where $z = \ka(\eps)(x - c(\eps)t)$. In this co-moving coordinate frame, we seek perturbed solutions to the PDE \eqref{S1:eq1} of the form
\begin{equation}\label{S4:eq2}
u(x,t) =  w^\eps(z) + \tilde{v}(z,t).
\end{equation}
Substituting \eqref{S4:eq2} into \eqref{S1:eq1}, the equation transforms to 
\begin{align*}
&\tilde{v}_t- \ka(\eps)c(\eps)(w^\eps)'(z) - \ka(\eps)c(\eps) \tilde{v}_z+\ka(\eps) \cM_{\ka(\eps)}(w^\eps)'(z) + \ka(\eps) \cM_{\ka(\eps)} \tilde{v}_z \\&+ \alpha(w^\eps(z)+\tilde{v})(\ka(\eps)(w^\eps)'(z)+\ka(\eps)\tilde{v}_{z})\\=& \ka^{2}(\eps)(w^\eps)''(z)+\ka^2(\eps) \tilde{v}_{zz} + r(w^\eps(z) + \tilde{v})(1 - w^\eps(z) - \tilde{v}).
\end{align*}
Since $w^\eps(z)$ satisfies \eqref{S3:eq4}, the equation governing the spatial perturbation $\tilde{v}(z,t)$ in the co-moving frame becomes
\begin{align*}
\tilde{v}_t - \ka(\eps)c(\eps) \tilde{v}_z + \ka(\eps) \cM_{\ka(\eps)} \tilde{v}_z + \alpha \ka(\eps) (w^\eps(z)\tilde{v})_z+\alpha\ka(\eps)\tilde{v}\tilde{v}_{z} = \ka^2(\eps) \tilde{v}_{zz} +r(1-2w^\eps(z))\tilde{v} -r\tilde{v}^{2}.
\end{align*}
Isolating the linear terms in $\tilde{v}$ and discarding the $\mathcal{O}(\tilde{v}^2)$ nonlinearities, we arrive at the following linearized equation
\begin{equation}\label{S4:eq3}
\tilde{v}_t = \ka^{2}(\eps) \tilde{v}_{zz} - \ka(\eps)(\cM_{\ka(\eps)} - c(\eps))\tilde{v}_z - \alpha \ka(\eps) (w^\eps(z) \tilde{v})_z + r(1 - 2w^\eps(z))\tilde{v}.
\end{equation}
Motivated by the study of spatially localized, finite-energy perturbations in the Galilean frame, we separate variables by assuming 
\begin{equation}\label{S4:eq4}
    \tilde{v}(z,t) = e^{\lambda t} v(z),
\end{equation} for a spectral parameter $\lambda \in \mathbb{C}$ and a spatial perturbation $v \in L^2(\R)$. Substituting \eqref{S4:eq4} into \eqref{S4:eq3} yields the following spectral eigenvalue problem
\begin{equation}\label{S4:eq5}
\lambda v = \mathcal{L}^\eps v,
\end{equation}
where for each $\eps$ the linearized operator $\mathcal{L}^\eps$ is defined as
\begin{align}\label{S4:eq6}
\begin{cases}
    \cL^{\eps}:L^{2}(\R)\to L^{2}(\R),\\ 
    \vspace{0.1cm} \mathcal{D}\left(\cL^{\eps}\right) =H^{\mu}(\R), \;\mu=\max\{2,\beta+1\},\\ \vspace{0.1cm}\cL^{\eps}\coloneqq p_{3}^{\eps}(z)\partial_{z}^{2} + p_{2}^{\eps}(z)\cM_{\ka(\eps)}\partial_{z}+p_{1}^{\eps}(z)\partial_{z} + p_{0}^{\eps}(z)\cI,
\end{cases}
\end{align}
with coefficients,
\begin{align}
   p_{3}^{\eps}(z) &\coloneqq \ka^{2}(\eps),\label{S4:eq7}\\
   p_{2}^\eps(z)&\coloneqq-\ka(\eps),\label{S4:eq8}\\
   p_{1}^{\eps}(z) &\coloneqq\ka(\eps)c(\eps) - \alpha\ka(\eps) w^{\eps}(z),\label{S4:eq9}\\
   p_{0}^{\eps}(z) &\coloneqq - \alpha \ka(\eps)(w^{\eps})'(z)+r(1 - 2w^{\eps}(z)),\label{S4:eq10}
\end{align}
for all $z\in\R$.

The viscous diffusion term $\ka^{2}(\eps)\partial_{z}^{2}$ is of order $2$ and the dispersive term $\cM_{\ka(\eps)}\partial_{z}$ is of pseudodifferential order $(\beta + 1)$ (by the growth condition in Assumption \ref{assumption2}).
Therefore, we define the domain of $\cL^{\eps}$ in \eqref{S4:eq6} as $\mathcal{D}\left(\cL^{\eps}\right) =H^{\mu}(\R), \;\mu=\max\{2,\beta+1\}$.
\vspace{10pt}
\begin{lemma}[Smooth and Periodic Coefficients]\label[lemma]{S4:lm1}
    The coefficients $p_{i}^{\eps}(z)$, $i=0,1,2,3$, and their derivatives are smooth, bounded and $2\pi$-periodic. In particular, $p_{i}^{\eps}(z)\in W^{1,\infty}(\R)$ for all $i=0,1,2,3$.
\end{lemma}
\begin{proof}
    $w^\eps(z)$ is $2\pi$-periodic for all $z\in\mathbb{T}$ and hence its first and second-order derivatives are $2\pi$-periodic for all $z\in\mathbb{T}$.
    The wave $w^\eps(z)$ is smooth in the spatial variable, and hence its first and second-order derivatives are also smooth. 
    Hence $w^\eps(z)$ is bounded and has a finite supremum on a compact domain. Hence $w^\eps(z)\in L^{\infty}(\mathbb{T})$. Similarly $(w^\eps)'(z),(w^\eps)''(z)\in L^{\infty}(\mathbb{T})$ and are bounded. Therefore, the periodic extension of $w^\eps$ from $\mathbb{T}$ to $\R$ via $w^\eps(z+2\pi)=w^\eps(z)$ and that of its first and second-order derivatives are $2\pi$-periodic, bounded and smooth. Hence the coefficients $p_{i}^{\eps}(z)$, $i=0,1,2,3$, and their derivatives are smooth, bounded and $2\pi$-periodic. In particular, $p_{i}^{\eps}(z)\in W^{1,\infty}(\R)$ for all $i=0,1,2,3$.
\end{proof}
Also, note that $\cL^{\eps}$ is densely defined on $L^{2}(\R)$. Now we prove that it is closed in $L^{2}(\R)$. We follow the proof of \citep[see][Lemma 3.1.2]{Kapitula2013} and incorporate changes due to the presence of the nonlocal multiplier operator in our problem. Refer to \cref{B} for definitions of the resolvent set and the spectrum of a closed, densely defined linear operator.
\vspace{10pt}
\begin{proposition}[$\cL^{\eps}$ is Closed]\label[proposition]{S4:prop1}
   The parameterized operator $\cL^{\eps}$ defined by \eqref{S4:eq6} with coefficients defined by \eqref{S4:eq7}-\eqref{S4:eq10} is a closed operator in $L^{2}(\R)$.
\end{proposition}
\begin{proof}
   Assume that $\{u_n\}_{n=1}^{\infty}\subset H^\mu(\R)$ converges to $u$ in $\|.\|_{L^{2}}$ and that $v_n\coloneqq\cL^{\eps}u_n$ converges to $v$ in $\|.\|_{L^{2}}$. We shall show that $u\in H^\mu(\R)$ and that $\cL^{\eps}u=v$. We decompose the linearized operator as $\cL^\eps=\cA^\eps+\cC^\eps$, where $\cA^\eps=p_{3}^{\eps}(z)\partial^{2}_{z}+p_{2}^{\eps}(z)\cM_{\ka(\eps)}\partial_z$ and $\cC^{\eps}=p_{1}^{\eps}(z)\partial_z+p_{0}^{\eps}(z)\mathcal{I}$. Suppose $\lambda\in\R^{+}$ then taking Fourier transform, the symbol of the operator $\cA^\eps-\lambda \mathcal{I}$ is 
   \begin{equation*}
   a(\xi) - \lambda:= -p_{3}^{\eps}(z) \xi^2+i \xi p_{2}^{\eps}(z) m(\ka(\eps) \xi)-\lambda.
   \end{equation*}
   By Assumption \ref{assumption2} and \cref{S4:lm1}, we see that there exist positive constants $M$, $\tilde{C_1}$ and $\tilde{C_2}$ such that 
\begin{equation*}
       |\operatorname{Re}(a(\xi)-\lambda)|=|p_{3}^{\eps}(z)  \xi^2+\lambda|\geq \tilde{C_1} \left(|\xi|^2+\lambda \right)  
   \end{equation*}
and
\begin{equation*}
       |\operatorname{Im}(a(\xi)-\lambda)|= |\xi p_{2}^{\eps}(z) m(\ka(\eps) \xi)|\geq \tilde{C_2} |\xi|^{\beta+1} , \quad  |\xi| > M.  
\end{equation*}
Therefore, for some positive constant $\tilde{C}$
\begin{align}
    |(a(\xi)-\lambda)|&\geq\max\left\{|\operatorname{Re}(a(\xi)-\lambda)|,|\operatorname{Im}(a(\xi)-\lambda)|\right\}\notag\\
    & \geq \tilde{C}|\xi|^\mu+\lambda >0, \quad  |\xi|> M.\label{S4:A}
\end{align}
For $|\xi|\leq M$, again using Assumption \ref{assumption2} and \cref{S4:lm1}, $\operatorname{Re}(a(\xi))$ and $\operatorname{Im}(a(\xi))$ are bounded and so $a(\xi)$ is bounded. Therefore, for sufficiently large $\lambda$, $|a(\xi)-\lambda|>0$.
Hence, for sufficiently large $\lambda$, the operator $\cA^\eps-\lambda \mathcal{I}$ is invertible and the symbol of its inverse $\frac{1}{a(\xi)-\lambda}$ is in $L^\infty(\R)$. By Plancherel's Theorem, the resolvent $(\cA^\eps-\lambda \mathcal{I})^{-1}\in B(L^{2}(\R))$. Using \eqref{S4:A}, we have 
\begin{equation*}
   \|(\cA^\eps-\lambda \mathcal{I})^{-1}\|_{B(H^{s},H^{s+\mu-1})}=\sup\limits_{\xi\in\R}\frac{\left(1+|\xi|^{2}\right)^\frac{\mu-1}{2}}{|a(\xi)-\lambda|}
   \leq \sup\limits_{\xi\in\R}\frac{\left(1+|\xi|^{2}\right)^\frac{\mu-1}{2}}{\tilde{C}|\xi|^\mu+\lambda}
   \leq C\lambda^{\frac{-1}{\mu}}.
\end{equation*}
Therefore, for all $s\geq 0$ and sufficiently large $\lambda\in\R^{+}$, the operator $\cA^\eps-\lambda \mathcal{I}$ is invertible with the resolvent $(\cA^\eps-\lambda \mathcal{I})^{-1}\in B(H^{s},H^{s+\mu-1})$ and satisfies 
\begin{equation}\label{S4:C}
\lim\limits_{\lambda\to \infty}\|(\cA^\eps-\lambda \mathcal{I})^{-1}\|_{B(H^{s},H^{s+\mu-1})}=0.
\end{equation} 
Since, $p_{1}^\eps(z)$, $p_{0}^\eps(z)\in W^{1,\infty}(\R)$ as proved in \cref{S4:lm1}, we can see that $\cC^{\eps}\in B(H^\mu, H^{\mu-1})$. We can rewrite the equation $\cL^{\eps}u_n = v_n$ as
\begin{equation*}
(\mathcal{I}+\mathcal{K}_{\lambda})u_n=(\cA^\eps-\lambda \mathcal{I})^{-1}(v_n-\lambda u_n),
\end{equation*}
where $\mathcal{K}_\lambda=(\cA^\eps-\lambda \mathcal{I})^{-1}\cC^\eps$. Since $\mu \geq 2$, $2\mu-2\geq \mu$, and $H^{2\mu-2}\subset H^{\mu}$, therefore, $\mathcal{K}_\lambda\in B(H^\mu)$ and from \eqref{S4:C}, $\lim\limits_{\lambda\to\infty}\|\mathcal{K}_\lambda\|_{B(H^\mu)}=0$. Therefore, we can choose $\lambda\in\R^+$, sufficiently large, such that by result on Neumann series \citep[see][\S 2.5]{Buffoni2003}, the operator $\mathcal{I}+\mathcal{K}_\lambda$ is boundedly invertible and 
\begin{equation}\label{S4:D}
u_n=(\mathcal{I}+\mathcal{K}_\lambda)^{-1}(\cA^\eps-\lambda \mathcal{I})^{-1}(v_n-\lambda u_n).
\end{equation}
From here we can infer that $u_n$ is a Cauchy sequence in $H^\mu(\R)$, and taking the limit $n\to\infty$ in \eqref{S4:D} yields
\begin{equation}\label{S4:E}
u=(\mathcal{I}+\mathcal{K}_\lambda)^{-1}(\cA^\eps-\lambda \mathcal{I})^{-1}(v-\lambda u)\in H^\mu(\R).
\end{equation}
From \eqref{S4:E} it easy to see that $\cL^\eps u=v$.
\end{proof}

\subsection{Floquet--Bloch Transform and Bloch Operators}
As proved before the linearized operator $\cL^\eps$ is a closed, densely defined linear operator with $2\pi$-periodic coefficients. Therefore, $\cL^\eps$ commutes with the translation operator defined in \eqref{S4:eq16} with $l=2\pi$. By \cref{S4:prop2} we conclude that $\cL^\eps$ has no $L^2$-point spectrum and \eqref{S4:eq15} establishes the fact that \begin{equation}\label{S4:eqA}
\sigma(\cL^\eps)_{|L^2(\mathbb{R})} = \sigma_{\text{ess}}(\cL^\eps)_{|L^2(\mathbb{R})}.
\end{equation}
Therefore, the spectrum of the linearized operator $\cL^\eps$ is purely essential and consists of continuous bands. A common method to characterize this essential spectrum is to use Floquet theory and parameterize the essential spectrum by the eigenvalues of an associated one-parameter family of linear operators.

However, since the linearized operator $\cL^\eps$ contains the nonlocal multiplier operator, we cannot decompose the spectral problem associated with $\cL^\eps$ into a system of ODEs. Hence, we cannot use classical Floquet theory to parameterize the essential spectrum of $\cL^\eps$. Instead, we use the Floquet--Bloch transform to parameterize the essential spectrum of $\cL^\eps$. We apply the Floquet--Bloch transform defined in \eqref{S4:eq17} to $\lambda v=\cL^\eps v$. Since $\cL^\eps:L^{2}(\R)\to L^{2}(\R)$ is a closed, densely defined operator with domain $\mathcal{D}(\mathcal{\cL^\eps})=H^\mu(\R)$ and $2\pi$-periodic coefficients we use \cref{C:prop1} to transform our spectral problem \eqref{S4:eq5} on the real line into a family of eigenvalue problems on a compact domain $\mathbb{T}$ as follows
\begin{align*}
 \mathcal{B}(\lambda v)( \tau,z) &= \mathcal{B}(\cL^\eps v)(\tau,z) \nonumber \\\implies
    \lambda(\mathcal{B}v)(\tau,z)&=(\cB\cL^\eps\cB^{-1})(\cB v)(\tau,z)\\
\implies\lambda(\mathcal{B}v)(\tau,z)&= \cL_{\tau}^{\eps}(\mathcal{B}v)(\tau,z)\nonumber \\\implies
    \lambda W(z) &= \cL_\tau^\eps W(z),
\end{align*}
where for each $\tau\in[-1/2,1/2)$, abusing notation, we define $(\mathcal{B}v)(\tau,\cdot)=W(\cdot)\in L^{2}_{\text{per}}([0,2\pi])$ which is $2\pi$-periodic in $z$ and using \eqref{S7:eq3} we arrive at the following equation
\begin{align*}
        \lambda W &=\ka^{2}(\eps)\left(\partial_{z} +i \tau\right)^{2}W- \ka(\eps)\left(\cM_{\ka(\eps),\tau}-c(\eps)\right)\left(\partial_{z}+i\tau\right) W -\alpha\ka(\eps)\left(\partial_{z}+i\tau\right) (w^\eps W) \\ & \qquad
        +r\left(1-2w^{\eps}\right)W,
\end{align*}
where
\begin{equation*}
\cM_{\ka(\eps),\tau}e^{inz}=m(\ka(\eps)(\tau+n))e^{inz}.
\end{equation*}
Therefore, we arrive at a family of spectral problems
\begin{equation}\label{S4:eq21}
\lambda W=\cL^{\eps}_{\tau}W,
\end{equation}
indexed by $\tau \in [-1/2, 1/2)$, with the Bloch operators defined as
\begin{align}\label{S4:eq22}
\begin{cases}
\cL^{\eps}_{\tau}:L^{2}_{\text{per}}([0,2\pi])\to L^{2}_{\text{per}}([0,2\pi]),\\ 
\vspace{0.1cm} \mathcal{D}\left(\cL^{\eps}_{\tau}\right) =H^{\mu}_{\text{per}}([0,2\pi]), \;\mu=\max\{2,\beta+1\},\\ \vspace{0.1cm}
\cL^{\eps}_{\tau}\coloneqq p_3^\eps(z)(\partial_z + i\tau)^2 + p_2^\eps(z)\mathcal{M}_{\ka(\eps),\tau}(\partial_z + i\tau) +p_1^\eps(z)(\partial_z + i\tau) +p_0^\eps(z)\cI,
\end{cases}
\end{align}
where the coefficients $p_{i}^{\eps}(z)$, $i=0,1,2,3$ are as defined in \eqref{S4:eq7}-\eqref{S4:eq10}.
\vspace{10pt}
\begin{lemma}[$\mathcal{L}^{\eps}_{\tau}$ is Closed]\label[lemma]{S4:lm10}The parameterized Bloch operator $\cL^{\eps}_{\tau}($indexed by $\tau \in [-1/2, 1/2))$ defined by \eqref{S4:eq22} with coefficients defined by \eqref{S4:eq7}-\eqref{S4:eq10} is a closed operator in $L^{2}_{\mathrm{per}}([0,2\pi])$.
\end{lemma}
\begin{proof}
This proof can be constructed like the proof of closedness of $\cL^\eps$ in $L^2(\R)$ as done in \cref{S4:prop1}. 
\end{proof}

\subsection{Characterization of the Spectrum of the Linearized Operator}
We now prove that for each $\tau \in [-1/2, 1/2)$ the corresponding Bloch operator as defined in \eqref{S4:eq22} has only the point spectrum.
\vspace{10pt}
\begin{lemma}[Spectrum-Characterization of $\cL^\eps_\tau$]\label[lemma]{S4:lem7}
For each $\tau \in [-1/2, 1/2)$, the spectrum of $\cL_\tau^\eps$ on $L_{\mathrm{per}}^2([0,2\pi])$ consists entirely of isolated eigenvalues of finite multiplicities. Thus, for each $\tau \in [-1/2,1/2)$, 
\begin{equation*}
\sigma(\cL_\tau^\eps)_{|L_{\mathrm{per}}^2([0,2\pi])}=\sigma_{\mathrm{pt}}(\cL_\tau^\eps)_{|L_{\mathrm{per}}^2([0,2\pi])}.
\end{equation*}
\end{lemma}
\begin{proof}
   Fixing $\tau \in [-1/2,1/2)$,  we decompose the Bloch operator as $\cL_\tau^\eps=\cA_\tau^\eps+\cC_\tau^\eps$, where $\cA_\tau^\eps=p_{3}^{\eps}(z)(\partial_z + i\tau)^2+p_{2}^{\eps}(z)\cM_{\ka(\eps),\tau}(\partial_z + i\tau)$ and $\cC_\tau^{\eps}=p_{1}^{\eps}(z)(\partial_z + i\tau)+p_{0}^{\eps}(z)\mathcal{I}$. For any $\lambda\in\R^+$ we can write
\begin{equation}\label{S4:eqF}
   \cL_\tau^\eps-\lambda\mathcal{I}=(\mathcal{I}+\mathcal{K}_{\lambda,\tau})(\cA_\tau^\eps-\lambda\mathcal{I}),
   \end{equation}
   where $\mathcal{K}_{\lambda,\tau}=\cC_\tau^\eps(\cA_\tau^\eps-\lambda\mathcal{I})^{-1}$. Using the Rellich--Kondrachov Theorem \citep[see][Proposition 3.4]{Taylor1996} and following the steps as in proof of \cref{S4:lm10}, we can prove that $\mathcal{K}_{\lambda,\tau}\in B(L_{\text{per}}^2([0,2\pi]))$. We can also show that
$\lim\limits_{\lambda\to\infty}\|\mathcal{K}_{\lambda,\tau}\|_{B(L_{\text{per}}^2([0,2\pi]))}=0$. Therefore, we can choose $\lambda_0\in\R^+$, sufficiently large, such that for all $\lambda>\lambda_0$ the operator $\mathcal{I}+\mathcal{K}_{\lambda,\tau}$ is boundedly invertible. Therefore, from \eqref{S4:eqF}, $\cL_\tau^\eps-\lambda\mathcal{I}$ is invertible and its inverse is bounded on $L_{\text{per}}^2([0,2\pi])$. Hence, for all $\lambda>\lambda_0$, $\lambda\in\rho(\cL_\tau^\eps)$ and $\rho(\cL_\tau^\eps)$ is nonempty.

Let $\lambda\in \rho(\cL_\tau^\eps)$. Then using the Rellich--Kondrachov Theorem we see that  $(\cL_\tau^\eps-\lambda\mathcal{I})^{-1}:L_{\text{per}}^2([0,2\pi])\to L_{\text{per}}^2([0,2\pi])$ is a compact operator. Using \citep[see][Chapter III, \S6, Theorem 6.29]{Kato1980} since the resolvent $(\cL_\tau^\eps-\lambda\mathcal{I})^{-1}$ of the closed operator $\cL_\tau^\eps$ is compact for some $\lambda\in\rho(\cL_\tau^\eps)$, we can conclude that the spectrum of $\cL_\tau^\eps$ on $L_{\text{per}}^2([0,2\pi])$  consists entirely of isolated eigenvalues with finite multiplicities. Such a spectrum of $\cL_\tau^\eps$ is a purely discrete point spectrum and for each $\tau \in [-1/2,1/2)$, we have $\sigma(\cL_\tau^\eps)_{|L_{\text{per}}^2([0,2\pi])}=\sigma_{\text{pt}}(\cL_\tau^\eps)_{|L_{\text{per}}^2([0,2\pi])}$. 
\end{proof}
We have already shown in \eqref{S4:eqA} that $\cL^\eps$ has only the essential spectrum (and no point spectrum), which is difficult to compute. Instead, we can now use the spectrum of the Bloch operators to characterize the essential spectrum of $\cL^\eps$.
\vspace{10pt}
\begin{proposition}[Spectrum-Characterization of $\cL^\eps$]\label[proposition]{S4:prop8}
Consider the closed operator $\cL^\eps : H^\mu(\mathbb{R}) \subset L^2(\mathbb{R}) \to L^2(\mathbb{R})$ and the associated closed Bloch operator for each $\tau \in [-1/2, 1/2)$ as $\cL_\tau^\eps : H_{\mathrm{per}}^\mu([0,2\pi]) \subset L_{\mathrm{per}}^2([0,2\pi])\to L_{\mathrm{per}}^2([0,2\pi])$. Then for any $\lambda \in \mathbb{C}$, the following statements are equivalent
\begin{enumerate}
    \item $\lambda \in \sigma(\cL^\eps)_{|L^2(\mathbb{R})}$.\label{S4:prop8.1}
    \item There exists $\tau \in [-1/2, 1/2)$ such that $\lambda \in \sigma(\cL_\tau^\eps)_{|L_{\mathrm{per}}^2([0,2\pi])}$.\label{S4:prop8.2}
    \item There exists a nonzero function $W$ of the form $W(z) = e^{i\tau z}v(z)$ for some $\tau \in [-1/2, 1/2)$ and $v \in H_{\mathrm{per}}^\mu([0,2\pi])$ such that $(\cL^\eps - \lambda I)W = 0$.\label{S4:prop8.3}
\end{enumerate}
\end{proposition}
\begin{proof}
Refer \citep[Proposition 3.1]{Johnson2013} and use \cref{S4:lem7}.
\end{proof}
From \cref{S4:prop8} we obtain the following characterization of the spectrum of the linearized operator $\cL^{\eps}$ in terms of the union of point spectra of the Bloch operators $\cL^\eps_\tau$
\begin{equation}\label{S4:eq23}  \sigma(\cL^\eps)_{|L^2(\mathbb{R})}=\bigcup \limits_{\tau \in [-1/2, 1/2)}\sigma(\cL_\tau^\eps)_{|L_{\text{per}}^2([0,2\pi])}.
\end{equation}

\subsection{The Spectral Stability Criterion}
We can now rigorously define the condition for the periodic traveling wave to be dynamically stable.
\vspace{10pt}
\begin{definition}[Spectral Stability]
The family of periodic traveling wave solutions $w^\eps$ of \eqref{S1:eq1} is said to be spectrally stable if the $L^2(\mathbb{R})$-spectrum of the linearized operator $\cL^\eps$ around the wave satisfies
\begin{equation}
\label{S4:eq24}\sigma(\cL^\eps)_{|L^2(\mathbb{R})} \subset \{ \lambda \in \mathbb{C} \mid \operatorname{Re}(\lambda) \leq 0 \}.
\end{equation}
Otherwise, we say the wave is spectrally unstable.
\end{definition}
In view of the Floquet--Bloch decomposition and \eqref{S4:eq23} the criterion \eqref{S4:eq24} is equivalent to the condition that for all $\tau \in [-1/2, 1/2)$, the point spectrum of the Bloch operators $\sigma(\cL_\tau^\eps)$ satisfy,
\begin{equation*}
    \bigcup \limits_{\tau \in [-1/2, 1/2)}\sigma(\cL_\tau^\eps)_{|L_{\text{per}}^2([0,2\pi])}\subset \{ \lambda \in \mathbb{C} \mid \operatorname{Re}(\lambda) \leq 0 \}.
\end{equation*}
We immediately obtain the following criterion for instability.
\vspace{10pt}
\begin{definition}[Spectral Instability Criterion]
For each $\eps \in (0, \eps_0)$, the periodic traveling wave $w^\eps$ is spectrally unstable if there exists a Bloch parameter $\tau_{0} \in [-1/2, 1/2)$ for which 
\begin{equation*}
\sigma_{\mathrm{pt}}(\cL_{\tau_{0}}^\eps)_{|L_{\mathrm{per}}^2([0,2\pi])} \cap \{ \lambda \in \mathbb{C} \mid \operatorname{Re}(\lambda) > 0 \} \neq \varnothing.
\end{equation*}
\end{definition}

\section{Spectral Instability}\label{section5}
In this section we shall prove the main result about the spectral instability of the small-amplitude periodic traveling waves $w^\eps(z)$ constructed earlier. The core mechanism of this instability is that these waves bifurcate from the trivial state $w \equiv 0$, which is inherently unstable owing to the presence of the monostable source term with $r > 0$.

\subsection{Relatively Bounded Perturbations}
We first decompose our Bloch operators $\cL_{\tau}^\eps$ into an unperturbed part (evaluated at the bifurcation point $\eps = 0$) and a small, relatively bounded perturbation. From \eqref{S1:eq4}, \eqref{S1:eq6}, and \eqref{S1:eq5} we obtain the critical parameters at the bifurcation point ($\eps = 0$) to be
\begin{equation*}
    w^0 = 0, \quad \ka(0) = \ka_0 = \sqrt{r},\quad\text{and}\quad c(0) = c_0 = m(\ka_0) = m(\sqrt{r}).
\end{equation*}
We decompose the Bloch operator $\cL_\tau^\eps : H_{\text{per}}^\mu([0,2\pi]) \to L_{\text{per}}^2([0,2\pi])$ as
\begin{equation*}
    \cL_\tau^\eps = \cL_\tau^0 + \cP_\tau^\eps,
\end{equation*}
where the unperturbed operator $\cL_\tau^0$ is given by
\begin{equation*}
    \cL_\tau^0 W \coloneqq \ka_0^2(\partial_z + i\tau)^2 W - \ka_0(\mathcal{M}_{\ka_0, \tau} - c_0)(\partial_z + i\tau) W + r W,
\end{equation*}
and the perturbation operator $\cP_\tau^\eps$, which collects all $\mathcal{O}(\eps)$ and $\mathcal{O}(\eps^2)$ corrections, is given by
\begin{align}\label{S5:eq3}
    \cP_\tau^\eps W &\coloneqq (\ka^2(\eps) - \ka_0^2)(\partial_z + i\tau)^2 W \nonumber \\
    &\quad - \left[ \ka(\eps)(\mathcal{M}_{\ka(\eps), \tau}-c(\eps)) -\ka_0(\mathcal{M}_{\ka_0, \tau} - c_0) \right](\partial_z + i\tau) W \nonumber\\
    &\quad - \alpha \ka(\eps)(\partial_z + i\tau)(w^\eps W) - 2r w^\eps W.
\end{align}
Therefore, we can rewrite the spectral problem \eqref{S4:eq21} as a perturbed spectral problem of the form
\begin{equation*}
    \lambda W = (\cL_\tau^0 + \cP_\tau^\eps) W, \qquad W \in H_{\text{per}}^\mu([0,2\pi])=\mathcal{D}(\cL_\tau^\eps).
\end{equation*}
Now we characterize the spectrum of $\cL^{0}_{\tau}$.
\vspace{10pt}
\begin{lemma}[Spectrum of $\mathcal{L}^{0}_{\tau}$]\label{lem:L0_spectrum}\label[lemma]{S5:lm2}
For any $\tau \in [-1/2, 1/2)$, the eigenvalues of the operator $\mathcal{L}^{0}_{\tau}$ are given explicitly by the sequence
\begin{equation}\label{S5:eq5}
\lambda_n^0(\tau) \coloneqq -\ka_0^2(n+\tau)^2 + r - i \ka_0(n+\tau)\Big( m(\ka_0(n+\tau)) - c_0 \Big), \quad n \in \Z.
\end{equation}
The corresponding eigenfunctions are $W_n(z) = e^{inz}$. Furthermore, the unique eigenvalue with the maximal real part occurs at $\tau = 0$ and $n = 0$, yielding $\lambda_0^0(0) = r > 0$.
\end{lemma}
\begin{proof}
From \cref{S4:lem7}, for any $\tau \in [-1/2, 1/2)$, the operator $\mathcal{L}^{0}_{\tau}$ has a purely discrete point spectrum. Applying $\mathcal{L}^{0}_{\tau}$ to $e^{inz}$ we obtain
\begin{equation*}
\mathcal{L}^{0}_{\tau}e^{inz}= \Big[ -\ka_0^2(n+\tau)^2 + r - i \ka_0(n+\tau)\big(m(\ka_0(n+\tau)) - c_0\big) \Big] e^{inz}.
\end{equation*}
The real part of each eigenvalue is $\operatorname{Re}(\lambda_n^0(\tau)) = -\ka_0^2(n+\tau)^2 + r= r(1 - (n+\tau)^2)$. 
Since $n \in \Z$ and $\tau \in [-1/2, 1/2)$, the quantity $(n+\tau)^2$ achieves its absolute minimum of $0$ strictly at $\tau=0$ and $n=0$. Thus, $\max\limits_{n,\tau} \operatorname{Re}(\lambda_n^0(\tau)) = \operatorname{Re}(\lambda_0^0(0)) = r(1-0) = r>0$.
\end{proof}
Next, we want to show that an isolated eigenvalue of $\mathcal{L}^{\eps}_{0}$ for $\eps$ sufficiently small is in neighborhood of unstable eigenvalue $r$ of $\mathcal{L}^{0}_{0}$. In order to do that, we must first show that the perturbation operator $\mathcal{P}^\eps_\tau$ is bounded relative to the unperturbed operator $\mathcal{L}^{0}_{\tau}$ with a vanishing relative bound as $\eps \to 0$ (refer \cref{S5:def3}). We start with a coercivity bound on $\mathcal{L}_\tau^0$.
\vspace{10pt}
\begin{lemma}[Coercivity Bound on $\cL_\tau^0$]\label[lemma]{S4:lemz}
There exists a constant $C>0$, independent of $W$, such that
\begin{equation}\label{S5:D}
 \|\mathcal{L}_{\tau}^{0} W\|_{L^2_\mathrm{per}([0,2\pi])} + \|W\|_{L^2_\mathrm{per}([0,2\pi])} \ge C\|W\|_{H^\mu_\mathrm{per}([0,2\pi])},
\end{equation}
where $\mu = \max\{2,\beta+1\}$.
\end{lemma}
\begin{proof}
We proceed by analyzing the operator in Fourier space. Expanding $W \in H^\mu_{\mathrm{per}}([0,2\pi])$ into a Fourier series, and applying $\mathcal{L}_\tau^0$ on $W$ we get
\begin{equation*}
    \mathcal{L}_\tau^0 W(z) = \sum_{n \in \mathbb{Z}} \lambda_n^0(\tau) \widehat{W}(n) e^{inz},
\end{equation*}
with $\lambda_n^0(\tau)$ as in \eqref{S5:eq5}, where
\begin{align*}
    \operatorname{Re}(\lambda_n^0(\tau)) &= -\kappa_0^2(n+\tau)^2 + r, \\
    \operatorname{Im}(\lambda_n^0(\tau)) &= -\kappa_0(n+\tau)\big(m(\kappa_0(n+\tau)) - c_0\big).
\end{align*}
 Now, the dominant term in the real part is clearly quadratic in $n$. Hence, there exist constants $N_1 > 0$ and $C_1 > 0$ such that for all $|n| \ge N_1$, we have
\begin{equation}\label{S4:eqZ1}
    |\operatorname{Re}(\lambda_n^0(\tau))| \ge C_1 n^2.
\end{equation}
Also, utilizing the Assumption \ref{assumption2} and noting the fact that the term involving $m$ dominates the constant $c_0$ in the imaginary part, we obtain that there exist constants $N_2 > 0$ and $C_2 > 0$ such that for all $|n| \ge N_2$,
\begin{equation}\label{S4:eqZ2}
    |\operatorname{Im}(\lambda_n^0(\tau))| \ge C_2 |n|^{\beta+1}.
\end{equation}
Now, the estimates \eqref{S4:eqZ1} and \eqref{S4:eqZ2} jointly imply that for $|n| \ge \max\{N_1, N_2\}$,
\begin{equation*}
    |\lambda_n^0(\tau)| \ge \max\left\{|\operatorname{Re}(\lambda_n^0(\tau))|, |\operatorname{Im}(\lambda_n^0(\tau))|\right\}\geq\max\left\{C_1 n^2, C_2 |n|^{\beta+1}\right\} \ge C_3 (1+|n|^2)^{\mu/2},
\end{equation*}
where $\mu = \max\{2, \beta+1\}$ and $C_3 > 0$ is a sufficiently small constant. For the remaining low-frequency modes, where $|n| < \max\{N_1, N_2\}$, the quantity $|\lambda_n^0(\tau)|^2 + 1$ is strictly positive. Since this involves only a finite number of modes, there exists a uniform constant $C_4 > 0$ such that for all $n \in \mathbb{Z}$, the following bound holds
\begin{equation}\label{S4:eqZ3}
    |\lambda_n^0(\tau)|^2 + 1 \ge C_4^2 (1+|n|^2)^{\mu}.
\end{equation}
Finally, using the inequality $(a+b)^2 \ge a^2+b^2$ for $a,b \ge 0$, we compute
\begin{align*}
    \big( \|\mathcal{L}_{\tau}^{0} W\|_{L^2_\mathrm{per}([0,2\pi])} + \|W\|_{L^2_\mathrm{per}([0,2\pi])}\big)^2 
    &\ge \|\mathcal{L}_{\tau}^{0} W\|_{L^2_\mathrm{per}([0,2\pi])}^2 + \|W\|_{L^2_{\mathrm{per}}([0,2\pi])}^2 \\
    &= \sum_{n \in \mathbb{Z}} |\lambda_n^0(\tau)|^2 |\widehat{W}(n)|^2 +  \sum_{n \in \mathbb{Z}} |\widehat{W}(n)|^2 \\
    &=  \sum_{n \in \mathbb{Z}} \big(|\lambda_n^0(\tau)|^2 + 1\big) |\widehat{W}(n)|^2,
\end{align*}
which by \eqref{S4:eqZ3} yields
\begin{align*}
    \big( \|\mathcal{L}_{\tau}^{0} W\|_{L^2_\mathrm{per}([0,2\pi])} + \|W\|_{L^2_\mathrm{per}([0,2\pi])}\big)^2 
    &\ge  C_4^2 \sum_{n \in \mathbb{Z}} (1+|n|^2)^{\mu}|\widehat{W}(n)|^2 \\
    &= C_4^2 \|W\|_{H^\mu_{\text{per}}([0,2\pi])}^2.
\end{align*}
Taking the square root of both sides and setting $C = C_4$ completes the proof.
\end{proof}
\begin{lemma}[Relative Boundedness]\label[Lemma]{S5:lm4}
For any fixed $\tau \in [-1/2, 1/2)$, the operator $\mathcal{P}^\eps_\tau$ is $\mathcal{L}^{0}_{\tau}$-bounded. Specifically, there exist nonnegative constants $a(\eps)$ and $b(\eps)$ such that for all $W \in H^\mu_{\mathrm{per}}([0,2\pi])$,
\begin{equation}\label{S5:eq7}
\|\mathcal{P}^\eps_\tau W\|_{L^2_\mathrm{per}([0,2\pi])} \le a(\eps)\|W\|_{L^2_\mathrm{per}([0,2\pi])} + b(\eps)\|\mathcal{L}^{0}_{\tau} W\|_{L^2_\mathrm{per}([0,2\pi])},
\end{equation}
and $\lim\limits_{\eps \to 0} a(\eps) = 0$ and $\lim\limits_{\eps \to 0} b(\eps) = 0$.
\end{lemma}
\begin{proof}
Using \eqref{S1:eq4}, \eqref{S1:eq6}, and \eqref{S1:eq5}, we can verify that the norm of the profile and its derivative satisfy
\begin{equation}\label{S5:A}
\|w^\eps\|_{L^\infty}=\mathcal{O}(\eps)\quad \text{ and }\quad \|(w^\eps)'\|_{L^\infty} = \mathcal{O}(\eps),   
\end{equation} 
and the parameters satisfy \begin{align}
    |\ka(\eps) - \ka_0| &= \mathcal{O}(\eps^2),\label{S5:B}\\ |c(\eps) - c_0| &= \mathcal{O}(\eps^2).\label{S5:C}
    \end{align} 
From \eqref{S5:B} we can obtain that
\begin{equation}\label{S5:E}
    \|(\ka^2(\eps) - \ka_0^2)(\partial_z + i\tau)^2 W\|_{L^2_\text{per}([0,2\pi])}\leq|\ka(\eps)^2 - \ka_0^2| \|W\|_{H^2_\text{per}([0,2\pi])} \le C_1 \eps^2 \|W\|_{H^\mu_\text{per}([0,2\pi])}.
\end{equation}
Using the smallness of the parameter shifts as depicted in \eqref{S5:B} and \eqref{S5:C}, and the fact that the symbol $m(\xi)$ is differentiable we can see that 
\begin{align}\label{S5:F}
\Big\| \Big[ \ka(\eps)(\cM_{\ka(\eps),\tau} - c(\eps)) - \ka_0(\cM_{\ka_0,\tau} - c_0) \Big](\partial_z + i\tau) W \Big\|_{L^2_\text{per}([0,2\pi])} &\le C_2 \eps^2 \|W\|_{H^{\beta+1}_{\text{per}}([0,2\pi])}\notag\\ &\le C_2 \eps^2 \|W\|_{H^\mu_\text{per}([0,2\pi])}.
\end{align}
Using \eqref{S5:A} the remaining terms of \eqref{S5:eq3} are bounded as
\begin{equation}\label{S5:G}
\|\alpha \ka(\eps) (\partial_z + i\tau)(w^\eps W) + 2r w^\eps W\|_{L^2_\text{per}([0,2\pi])} \le C_3 |\eps| \|W\|_{H^1_\text{per}([0,2\pi])} \le C_3 |\eps| \|W\|_{H^\mu_\text{per}([0,2\pi])}.
\end{equation}
Summing the bounds in \eqref{S5:E}, \eqref{S5:F}, and \eqref{S5:G} yields $\|\mathcal{P}^\eps_\tau W\|_{L^2_\text{per}([0,2\pi])} \le C_5 |\eps| \|W\|_{H^\mu_\text{per}([0,2\pi])}$. Using \eqref{S5:D} from \cref{S4:lemz}, we substitute $\|W\|_{H^\mu_\text{per}([0,2\pi])} \le \frac{1}{C}\left(\|\mathcal{L}_{\tau}^{0} W\|_{L^2_\text{per}([0,2\pi])} + \|W\|_{L^2_\text{per}([0,2\pi])}\right)$ to obtain
\begin{equation*}
\|\mathcal{P}^\eps_\tau W\|_{L^2_\text{per}([0,2\pi])} \le \frac{C_5 |\eps|}{C} \|W\|_{L^2_\text{per}([0,2\pi])} + \frac{C_5 |\eps|}{C} \|\mathcal{L}_{\tau}^{0} W\|_{L^2_\text{per}([0,2\pi])}.
\end{equation*}
Defining $a(\eps) = b(\eps) = C_5 |\eps| / C$, we see that both bounds vanish as $\eps \to 0$. 
\end{proof}

\subsection{Proof of Instability}
The results and definitions related to the perturbation theory used here are provided in \cref{D}. We first show the generalized convergence of our operators.
\vspace{10pt}
\begin{lemma}[Generalized Convergence of $\mathcal{L}^\eps_\tau$]\label[lemma]{S5:lm10}
For any $\tau \in [-1/2, 1/2)$, the family of Bloch operators $\mathcal{L}^\eps_\tau$ converges to $\mathcal{L}^{0}_{\tau}$ in the generalized sense as $\eps \to 0$.
\end{lemma}
\begin{proof}
By \eqref{S5:eq7} and \eqref{S5:eq8} from \cref{S5:thm7}, and the fact that $\mathcal{L}^\eps_\tau - \mathcal{L}^{0}_{\tau} = \mathcal{P}^\eps_\tau$ is relatively bounded with respect to $\mathcal{L}^{0}_{\tau}$ with relative bounds $a(\eps), b(\eps) \to 0$, we infer that for any $\tau \in [-1/2, 1/2)$, $\widehat{\delta}(\mathcal{L}^\eps_\tau,\mathcal{L}^{0}_{\tau})\to 0$ as $\eps\to 0$. 
\end{proof}
Next, we establish the isolation of the critical eigenvalue of the base operator.
\vspace{10pt}
\begin{lemma}[Isolation of the Critical Eigenvalue]\label[lemma]{S5:lm11}
Consider the unperturbed operator $\mathcal{L}^{0}_{0}$ evaluated at $\tau = 0$. The value $\lambda_0^0(0) = r$ is a simple, isolated eigenvalue of $\mathcal{L}^{0}_{0}$.
\end{lemma}
\begin{proof}
From Lemma \ref{S5:lm2}, we obtain the eigenvalues of $\mathcal{L}^{0}_{0}$ as $$\lambda_n^0(0) = -\ka_0^2 n^2 + r - i \ka_0 n(m(\ka_0 n) - c_0), \quad n\in\Z.$$
For $n = 0$, we have $\lambda_0^0(0) = r$. So, $\operatorname{Re}(\lambda_0^0(0)) = r>0$. For $n = \pm 1$, we have $\lambda_{\pm 1}^0(0) =-\ka_0^{2}+r=0$. So, $\operatorname{Re}(\lambda_{\pm 1}^0(0)) = 0$. For $|n| \geq 2$, we have $\operatorname{Re}(\lambda_n^0(0)) = r(1 - n^2)< -3r < 0$. The eigenvalue $\lambda_0^0(0) = r$ is strictly separated from the rest of the spectrum by a distance of at least $\sqrt{r^2 + (\operatorname{Im}(\lambda_1^0(0))^2}=\sqrt{r^{2}}=r>0$, proving it is an isolated point of the spectrum, and also has algebraic multiplicity of one. Thus the value $\lambda_0^0(0) = r$ is a simple, isolated eigenvalue of $\mathcal{L}^{0}_{0}$.
\end{proof}
We are now equipped to finalize the proof of spectral instability.
\vspace{10pt}
\begin{lemma}[Perturbation of the Unstable Eigenvalue]\label[lemma]{S5:lm12}
For $\eps > 0$ sufficiently small, the Bloch operator $\mathcal{L}^\eps_0$ possesses a simple eigenvalue $\lambda^\eps_0(0)$, satisfying $\operatorname{Re}(\lambda^\eps_0(0)) > 0$.
\end{lemma}
\begin{proof}
By Lemma \ref{S5:lm11}, $\lambda_0^0(0) = r > 0$ is a simple, isolated eigenvalue of $\mathcal{L}^{0}_{0}$ (so forms a finite system of eigenvalues). Let $\Gamma$ be a circle in the complex plane centered at $r>0$ with radius $r/2$. Thus, $\Gamma$ lies entirely in the open right-half of the complex plane and encloses only the eigenvalue $\lambda^{0}_0(0)$.
By \cref{S5:lm10}, $\mathcal{L}^\eps_0 \to \mathcal{L}^{0}_{0}$ in the generalized sense as $\eps\to 0$. Applying \cref{S5:thm9}, for $\eps$ sufficiently small, the curve $\Gamma$ encloses exactly one simple eigenvalue of $\mathcal{L}^\eps_0$, which we denote as $\lambda^\eps_0(0)$.
Since $\lambda^\eps_0(0)$ lies inside $\Gamma$, therefore $\operatorname{Re}(\lambda^\eps_0(0)) > 0$.
\end{proof}
\begin{proof}[Proof of \cref{S1:thm2}]
From \cref{S5:lm12}, we obtain that for $\tau=0$ the Bloch operator $\mathcal{L}^\eps_0$ possesses an eigenvalue $\lambda^\eps_0(0)$ with $\operatorname{Re}(\lambda^\eps_0(0)) > 0$ for all sufficiently small $\eps\in(0,\eps_0)$. From \eqref{S4:eq23}, $\lambda^\eps_0(0) \in \sigma(\mathcal{L}^\eps)_{|L^2(\R)}$. Thus,  
\begin{equation*}
\sigma(\cL^\eps)_{|L^2(\mathbb{R})} \cap \{ \lambda \in \mathbb{C} \mid \operatorname{Re}(\lambda) > 0 \} \neq \varnothing.
\end{equation*}
Thus, the periodic traveling wave $w^\eps$ is spectrally unstable with respect to perturbations in $L^{2}(\R).$ This rigorously concludes the proof of spectral instability.
\end{proof}

\section{Applications to Specific Dispersive Models}\label{section6}
We can apply the rigorous framework developed in the previous sections to various dispersive water wave models. By specifying the Fourier multiplier operator $\cM$ (and equivalently, its real-valued, even symbol $m(\xi)$), our results immediately yield the existence, precise asymptotic expansion, and spectral instability of small-amplitude periodic traveling waves for a wide variety of such systems. For all the following models, the critical wavenumber is fixed to be $\ka_0=\sqrt{r}$, and the Stokes expansions of the wave profile, wavenumber and the wave speed are as proved in \cref{S1:thm1}.

\subsection{The Korteweg--de Vries--Burgers--Fisher Equation}
The KdV--Burgers--Fisher equation was first proposed in \citep{Kocak2020} and is given by
\begin{equation}\label{S6:eq1}
u_t +u_{xxx} + \alpha u u_x = u_{xx} + r u(1-u).
\end{equation}
In this case, the multiplier operator is $\cM = \partial^2_{x}$ which has the polynomial symbol
\begin{equation*}
m(\xi) = -\xi^2.
\end{equation*}
It is easy to verify that the function $\xi\mapsto m(\xi)$ is real-valued and satisfies Assumption \ref{assumption1}. Assumption \ref{assumption2} holds true for $\beta=2$.
\vspace{10pt}
\begin{corollary}[KdV--Burgers--Fisher]\label[corollary]{S6:cor1}
For the KdV--Burgers--Fisher equation, the critical wave speed at the bifurcation point is
\begin{equation*}
c_0 = m(\sqrt{r}) = -r.
\end{equation*}
By \cref{S1:thm1} there exists an open neighborhood of the critical parameters $(\kappa_0, c_0)$ and a small parameter $\eps_0 > 0$ such that for each $\eps \in (0, \eps_0)$, there exists a unique (up to translation) nontrivial solution $u(x,t)=w(\kappa (\eps)(x-c(\eps)t))=:w^\eps(z)$ to \eqref{S6:eq1} with wave speed $c(\eps)$ and wavenumber $\kappa(\eps)$, where $w$ is $2\pi$-periodic and smooth in its argument. Furthermore, by \cref{S1:thm2}, these periodic traveling waves are spectrally unstable in $L^2(\R)$.
\end{corollary}
Our results for this model are consistent with those in \citep{Folino2024}.

\subsection{The Whitham--Burgers--Fisher Equation}
The Whitham--Burgers--Fisher equation is given by 
\begin{equation}\label{S6:eq2}
u_t + \mathcal{W} u_x + \alpha u u_x = u_{xx} + r u(1-u),
\end{equation}
where the multiplier operator $\cM$ is the Whitham multiplier operator $\mathcal{W}$ with symbol (normalized for gravity and depth) defined as
\begin{equation*}
m(\xi) = \sqrt{\frac{\tanh(\xi)}{\xi}}.
\end{equation*}
The function $\xi\mapsto m(\xi)$ is real-valued and satisfies Assumption \ref{assumption1}. Assumption \ref{assumption2} holds true for $\beta=-1/2$.
\vspace{10pt}
\begin{corollary}[Whitham--Burgers--Fisher]\label[corollary]{S6:cor2}
For the Whitham--Burgers--Fisher equation, the critical wave speed at the bifurcation point is given by
\begin{equation*}
c_0 = m(\sqrt{r})=\sqrt{\frac{\tanh(\sqrt{r})}{\sqrt{r}}}.
\end{equation*}
By \cref{S1:thm1} there exists an open neighborhood of the critical parameters $(\kappa_0, c_0)$ and a small parameter $\eps_0 > 0$ such that for each $\eps \in (0, \eps_0)$, there exists a unique (up to translation) nontrivial solution $u(x,t)=w(\kappa (\eps)(x-c(\eps)t))=:w^\eps(z)$ to \eqref{S6:eq2} with wave speed $c(\eps)$ and wavenumber $\kappa(\eps)$, where $w$ is $2\pi$-periodic and smooth in its argument. Furthermore, by \cref{S1:thm2}, these periodic traveling waves are spectrally unstable in $L^2(\R)$.
\end{corollary}

\subsection{The Benjamin--Ono--Burgers--Fisher Equation}
The BO--Burgers--Fisher equation is given by
\begin{equation}\label{S6:eq3}
u_t - \mathcal{H} u_{xx} + \alpha u u_x = u_{xx} + r u(1-u).
\end{equation}
Here, $\cM u_x = -\mathcal{H} u_{xx}$, where $\mathcal{H}$ is the Hilbert transform defined as 
\begin{equation*}
    \widehat{\mathcal{H}u}(\xi)=-i\operatorname{sgn}(\xi)\widehat{u}(\xi).
\end{equation*}
The Fourier symbol of the operator $\cM$ in this case is
\begin{equation*}
m(\xi) = -|\xi|.
\end{equation*}
The function $\xi\mapsto m(\xi)$ is real-valued and satisfies Assumption \ref{assumption1}. Assumption \ref{assumption2} holds true for $\beta=1$.
\vspace{10pt}
\begin{corollary}[BO--Burgers--Fisher]\label{S6:cor3}
For the BO--Burgers--Fisher equation, the critical wave speed at the bifurcation point is
\begin{equation*}
c_0 = m(\sqrt{r})=\sqrt{r}.
\end{equation*}
By \cref{S1:thm1} there exists an open neighborhood of the critical parameters $(\kappa_0, c_0)$ and a small parameter $\eps_0 > 0$ such that for each $\eps \in (0, \eps_0)$, there exists a unique (up to translation) nontrivial solution $u(x,t)=w(\kappa (\eps)(x-c(\eps)t))=:w^\eps(z)$ to \eqref{S6:eq3} with wave speed $c(\eps)$ and wavenumber $\kappa(\eps)$, where $w$ is $2\pi$-periodic and smooth in its argument. Furthermore, by \cref{S1:thm2}, these periodic traveling waves are spectrally unstable in $L^2(\R)$.
\end{corollary}

\subsection{The Fractional KdV--Burgers--Fisher Equation}
The fKdV--Burgers--Fisher equation utilizes a fractional derivative as
\begin{equation}\label{S6:eq5}
u_t - \Lambda^\beta u_x + \alpha u u_x = u_{xx} + r u(1-u),
\end{equation}
where $\beta \in (0, 2]$. The operator $\cM = -\Lambda^\beta$ is defined via the homogeneous fractional symbol
\begin{equation*}
m(\xi) = -|\xi|^\beta.
\end{equation*}
The function $\xi\mapsto m(\xi)$ is real-valued and satisfies Assumptions \ref{assumption1} and \ref{assumption2}.
\vspace{10pt}
\begin{corollary}[fKdV--Burgers--Fisher]\label[corollary]{S6:cor5}
For any fractional dispersion power $\beta \in (0, 2]$, the critical wave speed for the fKdV--Burgers--Fisher equation at the bifurcation point is given by
\begin{equation*}
c_0 = m(\sqrt{r})=-r^{\beta/2}.
\end{equation*}
By \cref{S1:thm1} there exists an open neighborhood of the critical parameters $(\kappa_0, c_0)$ and a small parameter $\eps_0 > 0$ such that for each $\eps \in (0, \eps_0)$, there exists a unique (up to translation) nontrivial solution $u(x,t)=w(\kappa (\eps)(x-c(\eps)t))=:w^\eps(z)$ to \eqref{S6:eq5} with wave speed $c(\eps)$ and wavenumber $\kappa(\eps)$, where $w$ is $2\pi$-periodic and smooth in its argument. Furthermore, by \cref{S1:thm2}, these periodic traveling waves are spectrally unstable in $L^2(\R)$.
\end{corollary}

\subsection{The Intermediate Long Wave--Burgers--Fisher Equation}
The ILW--Burgers--Fisher equation for a fluid of depth $\delta > 0$ incorporates the operator $\mathcal{T}_\delta$ as
\begin{equation}\label{S6:eq4}
u_t + \mathcal{T}_\delta u_{xx} + \alpha u u_x = u_{xx} + r u(1-u),
\end{equation}
where 
\begin{equation*}
    \widehat{\cT_{\delta} u}(\xi)=\left(-i \coth(\delta \xi)+\frac{i}{\delta\xi}\right)\widehat{u}(\xi).
\end{equation*}
Here $\cM u_x=\cT_{\delta}u_{xx}$ so that the symbol of the operator $\cM$ after normalization is given by
\begin{equation*}
m(\xi) = \xi \coth(\delta\xi) - \frac{1}{\delta}.
\end{equation*}
It is easy to check that the function $\xi\mapsto m(\xi)$ is real-valued and satisfies Assumption \ref{assumption1}. Assumption \ref{assumption2} holds true for $\beta=1$.
\vspace{10pt}
\begin{corollary}[ILW--Burgers--Fisher]\label[corollary]{S6:cor4}
For the ILW--Burgers--Fisher equation, the critical wave speed at the bifurcation point is
\begin{equation*}
c_0 = m(\sqrt{r})=\sqrt{r} \coth( \delta\sqrt{r}) - \frac{1}{\delta}.
\end{equation*}
By \cref{S1:thm1} there exists an open neighborhood of the critical parameters $(\kappa_0, c_0)$ and a small parameter $\eps_0 > 0$ such that for each $\eps \in (0, \eps_0)$, there exists a unique (up to translation) nontrivial solution $u(x,t)=w(\kappa (\eps)(x-c(\eps)t))=:w^\eps(z)$ to \eqref{S6:eq4} with wave speed $c(\eps)$ and wavenumber $\kappa(\eps)$, where $w$ is $2\pi$-periodic and smooth in its argument. Furthermore, by \cref{S1:thm2}, these periodic traveling waves are spectrally unstable in $L^2(\R)$.
\end{corollary}

\subsection{The Kawahara--Burgers--Fisher Equation}
The Kawahara--Burgers--Fisher equation is defined as
\begin{equation}\label{S6:eq6}
    u_t+u_{xxx}+\gamma u_{xxxxx}+\alpha uu_x= u_{xx}+ru(1-u).
\end{equation}
Here, $\cM=\partial^{2}_{x}+\gamma \partial_{x}^{4}$ with the symbol
\begin{equation*}
    m(\xi)=-\xi^{2}+\gamma \xi^{4}.
\end{equation*}
It is easy to check that the function $\xi\mapsto m(\xi)$ is real-valued and satisfies Assumption \ref{assumption1}. Assumption \ref{assumption2} holds true for $\beta=4$.
\vspace{10pt}
\begin{corollary}[Kawahara--Burgers--Fisher]\label[corollary]{S6:cor6}
For the Kawahara--Burgers--Fisher equation, the critical wave speed at the bifurcation point is
\begin{equation*}
c_0 = m(\sqrt{r})=-r +\gamma r^{2}.
\end{equation*}
By \cref{S1:thm1} there exists an open neighborhood of the critical parameters $(\kappa_0, c_0)$ and a small parameter $\eps_0 > 0$ such that for each $\eps \in (0, \eps_0)$, there exists a unique (up to translation) nontrivial solution $u(x,t)=w(\kappa (\eps)(x-c(\eps)t))=:w^\eps(z)$ to \eqref{S6:eq6} with wave speed $c(\eps)$ and wavenumber $\kappa(\eps)$, where $w$ is $2\pi$-periodic and smooth in its argument. Furthermore, by \cref{S1:thm2}, these periodic traveling waves are spectrally unstable in $L^2(\R)$.
\end{corollary}

\section*{Appendix}\label{appendix}
\appendix
\section{Transforming \texorpdfstring{$\mathcal{M}$}{M} to \texorpdfstring{$\mathcal{M}_\ka$}{M\_kappa}}\label[appendix]{A}
$\mathcal{M}$ is a pseudodifferential operator with symbol $m(\xi)$ such that \begin{equation*}
\mathcal{M}: H^s(\mathbb{R}) \to H^{s-\beta}(\mathbb{R}) \quad \text{for all } s\geq 0,
\end{equation*}is a bounded linear operator. This can be proved using the Assumption \ref{assumption2} as follows
\begin{align*}
    \|\cM f\|^{2}_{H^{s-\beta}(\R)}&=\int\limits_{-\infty}^{\infty}\left(1+|\xi|^{2}\right)^{s-\beta}|\widehat{\cM f}(\xi)|^{2}\, d\xi\\
    &=\int\limits_{-\infty}^{\infty}\left(1+|\xi|^{2}\right)^{s-\beta}|m(\xi) \widehat{f}(\xi)|^{2}\, d\xi\\
    &\leq C_1 \int\limits_{|\xi|\leq M}|\widehat{f}(\xi)|^{2}\, d\xi+C_2\int\limits_{|\xi|>M}\left(1+|\xi|^{2}\right)^{s-\beta}|\xi|^{2\beta} |\widehat{f}(\xi)|^{2}\, d\xi\\
    &\leq C_1\|f\|^2_{L^{2}(\R)}+C_2\int\limits_{|\xi|> M}\left(1+|\xi|^{2}\right)^{s}|\widehat{f}(\xi)|^{2}\, d\xi\\
    &\leq C_1\|f\|^{2}_{H^{s}(\R)}+C_2\|f\|^{2}_{H^{s}(\R)}\\
    &\leq C\|f\|^{2}_{H^{s}(\R)},
\end{align*}
where $C=\max\{C_1,C_2\}$. For $\xi_0$ fixed but arbitrary, we can see that
\begin{equation*}
    \widehat{e^{i\xi_0 x}}(\xi) = \delta(\xi-\xi_0),
\end{equation*}where $\delta$ is the Dirac delta function and therefore
\begin{equation*}
\widehat{\mathcal{M} e^{i\xi_0 x}}(\xi) = m(\xi)\delta(\xi-\xi_0).
\end{equation*}
Taking the inverse Fourier transform we obtain
\begin{equation*} \mathcal{M} e^{i\xi_0 x} = m(\xi_0)e^{i\xi_0 x}.
\end{equation*}
Now, $w(z)$ is $2\pi$-periodic and hence can be represented by a discrete Fourier series as
\begin{equation*}
w(z) = \sum_{n \in \mathbb{Z}} \widehat{w}(n) e^{inz}.
\end{equation*}
So,
\begin{equation*}
u(x,t) = \sum_{n \in \mathbb{Z}} \widehat{w}(n) e^{in(\ka(x-ct))} 
= \sum_{n \in \mathbb{Z}} \widehat{w}(n) e^{in\ka x} e^{-in\ka ct}.
\end{equation*}
Applying $\mathcal{M}$ we get
\begin{align*}
\mathcal{M}(u(x,t)) = \mathcal{M}\left( \sum_{n \in \mathbb{Z}} \widehat{w}(n) e^{in\ka x} e^{-in\ka ct} \right) 
= \sum_{n \in \mathbb{Z}} \widehat{w}(n) e^{-in\ka ct} \mathcal{M}(e^{in\ka x}) 
= \sum_{n \in \mathbb{Z}} \widehat{w}(n) m(\ka n) e^{in(\ka(x-ct))}.
\end{align*}
Therefore, we see how $\mathcal{M}$ transforms into a parametrized periodic multiplier $\mathcal{M}_{\ka}$ where
\begin{equation*}
\mathcal{M}_{\ka} : H^s_\text{per}(\mathbb{T}) \to H^{s-\beta}_\text{per}(\mathbb{T}) \quad \text{for all } \ka>0 \text{ and for all } s\geq 0,
\end{equation*}
defined as 
\begin{equation*}
\widehat{\mathcal{M}_\ka w}(n) = m(\ka n)\widehat{w}(n),
\end{equation*}
is a bounded linear operator. This can be proved using the discrete Fourier norm for $H^s_\text{per}(\mathbb{T})$ similar to the proof of boundedness for $\cM$.

\section{Resolvent Set and Spectrum}\label[appendix]{B}
Let $X, Y$ be Banach spaces and $\mathcal{A}:X\to Y$ be a closed, densely defined linear operator with domain $\mathcal{D}(\mathcal{A})\subseteq X$. The resolvent set and spectrum of $\mathcal{A}$  are defined as follows 
\vspace{10pt}
\begin{definition}[Resolvent Set and Spectrum of $\cA$]\label[definition]{S4:def1} \hfill
\begin{enumerate}
    \item The resolvent set $\rho(\cA)_{|X}$ is defined as
\begin{equation*}
    \rho(\cA)_{|X} = \{ \lambda \in \mathbb{C} \mid (\cA-\lambda \cI)^{-1} \text{ exists and is a bounded operator} \}.
\end{equation*}
$(\cA-\lambda \cI)^{-1}$ is known as the resolvent of $\cA$.
    \item The spectrum $\sigma(\cA)_{|X}$ is defined as
\begin{equation*}
    \sigma(\cA)_{|X}= \mathbb{C} \setminus \rho(\cA)_{|X}.
\end{equation*}
The spectrum can be decomposed into two sets
\begin{enumerate}
    \item The essential spectrum $\sigma_{\mathrm{ess}}(\cA)_{|X}$ is defined as
\begin{equation*}
\sigma_{\mathrm{ess}}(\cA)_{|X} =\{\lambda \in \mathbb{C} \mid(\cA-\lambda \cI) \text{ is either not Fredholm or has index different from } 0\}.
\end{equation*}
    \item The point spectrum $\sigma_{\mathrm{pt}}(\cA)_{|X}$ is defined as
\begin{equation*}
\sigma_{\mathrm{pt}}(\cA)_{|X} = \{ \lambda \in \mathbb{C} \mid (\cA-\lambda \cI ) \text{ is Fredholm with index } 0 \text{, but is not invertible} \}.
\end{equation*}
\end{enumerate}
\end{enumerate}
\end{definition}
\begin{lemma}\label[lemma]{lem:lm3}
$\sigma(\cA)_{|X} = \sigma_{\mathrm{ess}}(\cA)_{|X} \cup \sigma_{\mathrm{pt}}(\cA)_{|X}$.   
\end{lemma}
\begin{proof}
  It suffices to prove that $$(\sigma_{\text{ess}}(\cA)_{|X})^c \cap (\sigma_{\text{pt}}(\cA)_{|X})^c \subset \rho(\cA)_{|X}.$$Let $\lambda\in(\sigma_{\text{ess}}(\cA)_{|X})^c \cap (\sigma_{\text{pt}}(\cA)_{|X})^c$, then $(\cA-\lambda \cI )$ is an invertible Fredholm operator with index zero. By the Closed Graph Theorem, its inverse $(\cA-\lambda \cI)^{-1}$ is a bounded operator. Therefore, $\lambda \in \rho(\cA)_{|X}$.
\end{proof}
To read more about the resolvent set and spectrum of a closed, densely defined linear operator, refer \citep{Kapitula2013}.
\vspace{10pt}
\begin{proposition}\label[proposition]{S4:prop2}
Let $\cA:L^{2}(\R)\to L^{2}(\R)$ be a closed, densely defined operator which commutes with the translation operator $T : L^2(\mathbb{R}) \to L^2(\mathbb{R})$ defined by
\begin{equation}\label{S4:eq16}
    Tu(z) \coloneqq u(z + l),
\end{equation}
for some fixed $l\in\R\setminus\{0\}$.
Then the operator $\cA$ has no point spectrum and hence
\begin{equation}\label{S4:eq15}
\sigma(\cA)_{|L^2(\mathbb{R})} = \sigma_{\mathrm{ess}}(\cA)_{|L^2(\mathbb{R})}.
\end{equation}
\end{proposition}
\begin{proof}
Suppose, for the sake of contradiction, that $\lambda \in \sigma_{\text{pt}}(\cA)_{|L^{2}(\R)}$. Then $\lambda\in\C$ such that $(\cA-\lambda \cI)$ is a Fredholm operator with index 0, and $(\cA-\lambda \cI)$ is not invertible. Therefore, the eigenspace $E_\lambda = \ker(\cA-\lambda \cI)\neq \{0\}$ and is finite-dimensional. $\cA$ commutes with $T$ for any $u\in L^{2}(\R)$, that is, $$\cA T(u)=T\cA (u),\qquad u\in L^{2}(\R).$$ Hence, the eigenspace $E_\lambda$ is invariant under $T$. Indeed, if $u \in E_\lambda$ meaning $(\cA-\lambda \cI)u = 0$, then $(\cA-\lambda \cI )(Tu)=\cA Tu-\lambda Tu=T\cA u- T(\lambda u)= T((\cA-\lambda\cI)u)= T(0) = 0$, so $Tu\in E_\lambda$. Since $T$ is a unitary operator on $L^{2}(\R)$, its restriction $\tilde{T}$ to the finite-dimensional space $E_\lambda$ is a finite-dimensional unitary operator. Over the complex numbers, any finite-dimensional unitary operator has at least one eigenvector. Therefore, $\tilde{T}$ has an eigenvalue $\alpha\in\C$ with $|\alpha|=1$ satisfying 
\begin{equation*}
    \tilde{T}u = \alpha u
\end{equation*}
for some $u\in E_\lambda$.
By \eqref{S4:eq16}, this means $u(z+l) = \alpha u(z)$ for almost every $z$. Taking the absolute value of both sides and using $|\alpha|=1$, we obtain
\begin{equation*}
    |u(z+l)| = |u(z)|\qquad \text{a.e. } z.
\end{equation*}
This implies the function $v(z) = |u(z)|^2$ is $l$-periodic. Now, we evaluate the $L^2(\mathbb{R})$ norm of $u$ as
\begin{equation*}
\|u\|_{L^2(\mathbb{R})}^2 = \int_{-\infty}^{\infty} |u(z)|^2 dz= \sum_{n \in \mathbb{Z}} \int_{l n}^{l(n+1)} |u(z)|^2 dz.
\end{equation*}
Because $|u(z)|^2$ is $l$-periodic, the integral over any interval of length $l$ is the same constant $a = \int_0^{l} |u(z)|^2 dz$. Thus
\begin{equation*}
\|u\|_{L^2(\mathbb{R})}^2 = \sum_{n \in \mathbb{Z}} a.
\end{equation*}
Since $u \in L^2(\mathbb{R})$, we must have $\|u\|_{L^2(\mathbb{R})}^2 < \infty$. This infinite sum $\sum\limits_{n \in \mathbb{Z}} a$ converges if and only if $a = 0$. But if $a = 0$, then $ u(z) = 0$ almost everywhere. This contradicts the fact that $u$ is a nonzero eigenvector. Thus, $\sigma_{\text{pt}}(\cA)_{|L^{2}(\R)}$ is empty, and using \cref{lem:lm3}, we can imply that $\sigma(\cA)_{|L^2(\mathbb{R})} = \sigma_{\text{ess}}(\cA)_{|L^2(\mathbb{R})}$.
\end{proof}

\section{Floquet--Bloch Transform}\label[appendix]{C}
In this section, we develop the Floquet--Bloch theory.
\vspace{10pt}
\begin{definition}[Floquet--Bloch Transform]
Let $\tau \in [-1/2, 1/2)$ be the Bloch parameter (or the Floquet exponent). For a function $f\in\mathcal{S}(\mathbb{R})$, the Floquet--Bloch transform $\mathcal{B}f$ is defined as
\begin{equation}\label{S4:eq17}
    (\mathcal{B}f)(\tau,z) = \sum_{n \in \mathbb{Z}} e^{-i\tau(z+2\pi n)} f(z+2\pi n).
\end{equation}
\end{definition}
Note that for any $\tau \in [-1/2, 1/2)$, the function $(\mathcal{B}f)( \tau,.)$ is $2\pi$-periodic in $z$. 
\vspace{10pt}
\begin{lemma}[Unitary isomorphism on $L^2(\R)$]\label[lemma]{S4:G}
    The map $\mathcal{B}$ uniquely extends to a unitary isomorphism $\mathcal{B} : L^2(\mathbb{R}) \to L^2([-1/2, 1/2); L^{2}_{\mathrm{per}}([0, 2\pi]))$.
\end{lemma}
\begin{proof}
Let $\mathcal{H}=L^2([-1/2, 1/2); L^{2}_{\text{per}}([0, 2\pi]))$. For any $f\in \mathcal{S}(\mathbb{R})$, we compute
\begin{align*}
\|\mathcal{B}f\|_{\mathcal{H}}^2&=\int_{-1/2}^{1/2}\|(\mathcal{B}f)(\tau,z)\|^{2}_{L^{2}_\text{per}([0,2\pi])}\, d\tau\\&=\int_{-1/2}^{1/2} \int_{0}^{2\pi} \Big| \sum_{n \in \mathbb{Z}} e^{-i\tau(z + 2\pi n)} f(z + 2\pi n) \Big|^2 \,dz \, d\tau\\ &= \int_{-1/2}^{1/2}\int_{0}^{2\pi} \sum_{n\in\Z} \sum_{m\in\Z}f(z + 2\pi n) \overline{f(z + 2\pi m)} e^{i\tau 2\pi (m - n)} \, dz\, d\tau\\&= \sum_{n\in\Z} \sum_{m\in\Z} \int_{0}^{2\pi} f(z + 2\pi n) \overline{f(z + 2\pi m)} \left[ \int_{-1/2}^{1/2} e^{i\tau 2\pi (m - n)} \, d\tau \right]\, dz.
\end{align*}
By the orthogonality of the Fourier basis on $[-1/2, 1/2)$, the bracketed integral is exactly the Kronecker delta $\delta_{nm}$. Therefore,
\begin{align*}
\|\mathcal{B}f\|_{\mathcal{H}}^2 &= \sum_{n \in \mathbb{Z}} \int_{0}^{2\pi} |f(z + 2\pi n)|^2 \, dz\\ &=\sum_{n \in \mathbb{Z}} \int_{2\pi n}^{2\pi(n+1)} |f(x)|^2 \, dx\\&= \int_{-\infty}^{\infty} |f(x)|^2 \, dx \\&= \|f\|_{L^2(\mathbb{R})}^2.
\end{align*}
To prove surjectivity, we can show that $\mathcal{B}$ admits an explicit inverse known as the inverse Bloch transform, $\mathcal{B}^{-1}:L^2([-1/2, 1/2); L^{2}_{\text{per}}([0, 2\pi]))\to L^{2}(\R)$ defined by 
\begin{equation}\label{S7:eq1}
    (\mathcal{B}^{-1}W)(z)=\int_{-1/2}^{1/2} e^{i\tau z} W(\tau, z) \,d\tau
\end{equation} 
for any $W(\tau,z)$ in a dense subspace of $\mathcal{H}$. We can show similarly that $\mathcal{B}^{-1}$ also preserves norm. Then using Bounded Linear Transformation Theorem \citep[see][Theorem I.7]{ReedI}, $\mathcal{B}$ uniquely extends to a unitary isomorphism $\mathcal{B} : L^2(\mathbb{R}) \to L^2([-1/2, 1/2); L^{2}_{\text{per}}([0, 2\pi]))$.
\end{proof}
Let $\cA:L^{2}(\R)\to L^{2}(\R)$ be a closed, densely defined operator with $2\pi$-periodic coefficients. Then, for such an operator $\mathcal{A}$, the following proposition is proved.
\vspace{10pt}
\begin{proposition}[Bloch operator]\label[proposition]{C:prop1}
    Let $\cA:L^{2}(\R)\to L^{2}(\R)$ be a closed, densely defined operator with $2\pi$-periodic coefficients and domain $\mathcal{D}(\cA)=H^\mu(\R)$. Let $\mathcal{B}$ from $L^2(\mathbb{R})$ to $L^2([-1/2, 1/2); L^{2}_{\mathrm{per}}([0, 2\pi]))$ be the Floquet--Bloch transform as defined in \eqref{S4:eq17} and $\mathcal{B}^{-1}$ from $L^2([-1/2, 1/2); L^{2}_{\mathrm{per}}([0, 2\pi]))$ to $L^{2}(\R)$ be the inverse Bloch transform as defined by \eqref{S7:eq1}. Then for every $f\in L^2([-1/2, 1/2); H^\mu_{\mathrm{per}}([0, 2\pi]))$ the operator $\cB\cA\cB^{-1}$ from $L^2([-1/2, 1/2); L^{2}_{\mathrm{per}}([0, 2\pi]))$ to $L^2([-1/2, 1/2); L^{2}_{\mathrm{per}}([0, 2\pi]))$ decomposes for each $\tau \in [-1/2,1/2)$ as    
    \begin{equation*}
        (\cB\cA\cB^{-1}f)(\tau,z)=\cA_\tau f(\tau,z) \qquad \text{for all } z\in[0,2\pi],
    \end{equation*}
    where for each $\tau \in [-1/2,1/2)$ the Bloch operator $\cA_\tau:L^{2}_{\mathrm{per}}([0, 2\pi])\to L^{2}_{\mathrm{per}}([0, 2\pi])$ with $\mathcal{D}(\cA_\tau)=H^{\mu}_{\mathrm{per}}([0, 2\pi])$ is defined as 
\begin{equation}\label{S7:eq3}
\cA_\tau=e^{-i\tau z}\cA e^{i\tau z}.   
\end{equation}
\end{proposition}
\begin{proof}
   Applying $\cB^{-1}$ to every $f$ in a dense subspace of $L^2([-1/2, 1/2); L^{2}_{\mathrm{per}}([0, 2\pi]))$, we get
\begin{equation*}
    (\mathcal{B}^{-1}f)(z)=\int_{-1/2}^{1/2} e^{i\tau' z} f(\tau', z) \,d\tau'.
\end{equation*}
Applying $\cA$ to both sides, we obtain
\begin{equation*}
    (\cA\cB^{-1}f)(z)=\int_{-1/2}^{1/2} \cA\left(e^{i\tau' z} f(\tau', z)\right) \,d\tau'.
\end{equation*}
Finally applying the Floquet--Bloch transform $\cB$ on both sides, yields
\begin{align*}
(\cB\cA\cB^{-1}f)(\tau,z)&=\sum\limits_{n\in\Z}e^{-i\tau(z+2\pi n)}(\cA\cB^{-1}f)(z+2\pi n)\\
&=\sum\limits_{n\in\Z}e^{-i\tau(z+2\pi n)}\int_{-1/2}^{1/2} \cA\left(e^{i\tau' (z+2\pi n)} f(\tau', z+2\pi n)\right) \,d\tau'.
\end{align*}
 Since $f(\tau',.)\in H^\mu_\text{per}([0,2\pi])$, hence $f(\tau', z+2\pi n)=f(\tau',z)$ for all $z\in[0,2\pi]$. Therefore,
\begin{align*}
    (\cB\cA\cB^{-1}f)(\tau,z)&=e^{-i\tau z}\int_{-1/2}^{1/2}\left(\sum\limits_{n\in\Z}e^{i(\tau'-\tau)2\pi n}\right)\cA\left(e^{i\tau' z} f(\tau', z)\right) \,d\tau'\\
    &=e^{-i\tau z}\int_{-1/2}^{1/2}\delta(\tau'-\tau)\cA\left(e^{i\tau' z} f(\tau', z)\right) \,d\tau'\\
    &=e^{-i\tau z}\cA\left(e^{i\tau z} f(\tau, z)\right)\\
    &=\cA_\tau f(\tau,z),
\end{align*}
for all $z\in[0,2\pi]$. Therefore for every $f\in L^2([-1/2, 1/2); H^\mu_{\text{per}}([0, 2\pi]))$ following equality holds \begin{equation*}
    (\cB\cA\cB^{-1}f)(\tau,z)=\cA_\tau f(\tau,z) \qquad \text{for all } z\in[0,2\pi],
\end{equation*}where for each $\tau \in [-1/2,1/2)$  we can define the operator $\cA_\tau:L^{2}_{\text{per}}([0, 2\pi])\to L^{2}_{\text{per}}([0, 2\pi])$ with $\mathcal{D}(\cA_\tau)=H^{\mu}_{\text{per}}([0, 2\pi])$ as   
\begin{equation*}
    \cA_\tau=e^{-i\tau z}\cA e^{i\tau z}.
\end{equation*}
\end{proof}

\section{Abstract Perturbation Theory}\label[appendix]{D}
\begin{definition}[Relatively Bounded Operator]\citep[see][]{Kato1980}\label[definition]{S5:def3}
    Let $\mathcal{T}$ and $\mathcal{A}$ be operators defined on a Banach space $X$ with $\mathcal{D}(\mathcal{T})\subseteq\mathcal{D}(\mathcal{A})$ and 
\begin{equation}\label{S5:eq6} \|\mathcal{A}u\|\leq a\|u\|+b\|\mathcal{T}u\|, \qquad u\in\mathcal{D}(\mathcal{T}),
\end{equation}
where $a,b$ are nonnegative constants. Then we say that $\mathcal{A}$ is relatively bounded with respect to $\mathcal{T}$ or simply $\mathcal{T}$-bounded. The greatest lower bound $b_0$ of all possible constants $b$ in \eqref{S5:eq6} will be called the relative bound of $\mathcal{A}$ with respect to $\mathcal{T}$ or simply the $\mathcal{T}$-bound of $\mathcal{A}$.  
\end{definition}
Now we review the perturbation theory results. Refer \citep{Kato1980} for details.
\vspace{10pt}
\begin{definition}[Gap Between Closed Operators]\label{S5:def5} 
Let $X$ and $Y$ be Banach spaces and $\mathcal{C}(X,Y)$ be the set of all closed operators from $X$ to $Y$. If $\mathcal{A}, \mathcal{T} \in\mathcal{C}(X,Y)$, their graphs $\cG(\cA),\cG(\cT)$ are closed linear manifolds of the product space $X\times Y$. Set 
\begin{equation*}
\delta(\cA,\cT)=\delta(\cG(\cA),\cG(\cT)),
\end{equation*} where 
\begin{align*}
\delta(\cG(\cA),\cG(\cT))=
\begin{cases}
\sup\limits_{\substack{x\in\cG(\cA)\\\|x\|=1 }} \operatorname{dist}(x,\cG(\cT)), \quad &\cG(\cT)\neq \{0\}, \\
0, \quad &\cG(\cT)= \{0\},
\end{cases}
\end{align*}
and 
\begin{equation*}
    \widehat{\delta}(\cA,\cT)=\widehat{\delta}(\cG(\cA),\cG(\cT))=\max\left\{\delta(\cG(\cA),\cG(\cT)),\delta(\cG(\cT),\cG(\cA))\right\}.
\end{equation*}
$\widehat{\delta}(\cA,\cT)$ is called the gap between the closed operators $\cA$ and $\cT$.  
\end{definition}
\vspace{10pt}
\begin{definition}[Generalized Convergence]\label{S5:def6}
    A family of operators $\cT^\eps\in\mathcal{C}(X)$ converges to $\cT\in\mathcal{C}(X)$ in the generalized sense iff $\widehat{\delta}(\cT^\eps,\cT)\to 0$ as $\eps\to 0$.
\end{definition}
\vspace{10pt}
\begin{theorem}[Gap Between Relatively Bounded Closed Operators]\label{S5:thm7}
    Let $\cT\in\cC(X)$ and let $\cA$ be $\cT$-bounded with relative bound $b$ less than $1$, so that we have the inequality \eqref{S5:eq6} with $b < 1$. Then $\mathcal{E}= \cT+ \cA\in\cC(X)$ and
\begin{equation}\label{S5:eq8}
    \widehat{\delta}(\mathcal{E},\cT)\leq (1-b)^{-1}(a^{2}+b^{2})^{1/2}.
\end{equation}
\end{theorem}
\begin{proof}
    Refer \citep[Chapter IV, \S2, Theorem 2.14]{Kato1980}.
\end{proof}

\begin{definition}[Finite System of Eigenvalues]\label{S5:def8}
    It is a finite collection $\Sigma $ of eigenvalues of $\cT$ such that $\Sigma$ is isolated from the rest of the spectrum, and each eigenvalue is of finite algebraic multiplicity. The total algebraic multiplicity of the system is the sum of the algebraic multiplicity of each individual eigenvalue.
\end{definition}
\vspace{10pt}
\begin{theorem}[Continuity of a Finite System of Eigenvalues]\label{S5:thm9} Let 
$X$ be a complex Banach space and $\cT^\eps,\cT\in\cC(X)$. Let $\sum(\cT)=\{\lambda_1,\cdots,\lambda_r\}$ be a finite system of eigenvalues of $\cT$ with algebraic multiplicity $m$. Let $\Gamma$ be a rectifiable, simple closed curve such that $$\sum(\cT)\subset\operatorname{int}\Gamma
,\qquad \sigma(\cT)\setminus\sum(\cT)\subset\operatorname{ext}\Gamma, \qquad \Gamma\subset\rho(\cT).$$ If $\cT^\eps\to\cT$ in the generalized sense as $\eps\to 0$ then for sufficiently small $\eps$, $\cT^\eps$ also has a finite system of eigenvalues $\sum(\cT^\eps)$ such that 
$$\sum(\cT^\eps)\subset \operatorname{int}\Gamma,$$ with total algebraic multiplicity $m$. In particular, if $\sigma(\cT)=\{\lambda_0\}$ consists of a simple isolated eigenvalue, then for sufficiently small
$\eps$, the operator $\cT^\eps$ has exactly one simple eigenvalue
inside any sufficiently small disk around $\lambda_0$.
\end{theorem}
\bibliographystyle{amsalpha}
\bibliography{Bibliography}
\end{document}